\documentclass[oneside]{article}
\usepackage{amsfonts}
\usepackage[a4paper, total={6in, 8in}]{geometry}
\usepackage{bm}
\usepackage{amsmath}
\usepackage{hyperref}
\usepackage{multicol}
\usepackage{authblk}
\usepackage{amsthm}
\usepackage{graphicx}
\usepackage{float}

\usepackage[font=small]{caption}

\usepackage{pgfplots}
\usepackage{subcaption}
\usepackage{pdflscape}
\usepackage{array}
\usepackage{epstopdf}
\usepackage{cleveref}
\usepackage{amsmath}
\usepackage{amsfonts}
\usepackage{amssymb}
\usepackage{pgfplotstable}
\usepackage{pgfplots}
\usepackage{enumerate}

\usetikzlibrary{positioning}
\usetikzlibrary{shapes.callouts}
\mathchardef\mhyphen="2D

\newcommand{\vpinn}{{V\kern-0.20em P}}

\usepackage{authblk}
\usepackage{amsthm}

\newtheorem{theorem}{Theorem}[section]

\newtheorem{proposition}[theorem]{Proposition}

\newtheorem{lemma}[theorem]{Lemma}

\usepackage{tikz-network}
\usepackage{tikz}

\usepgfplotslibrary{external} 
\begin{document}

\title{A Deep Fourier Residual Method for solving PDEs using Neural Networks}
\author[1]{Jamie M. Taylor}
\author[2,3,4]{David Pardo}
\author[2,5]{Ignacio Muga}
\affil[1]{Department of Quantitative Methods, CUNEF University, Madrid, Spain}
\affil[2]{Basque Center for Applied Mathematics (BCAM), Bilbao, Bizkaia, Spain}
\affil[3]{University of the Basque Country (UPV/EHU), Leioa, Spain}
\affil[4]{Ikerbasque (Basque Foundation for Sciences), Bilbao, Spain}
\affil[5]{Instituto de Matem\'aticas, Pontificia Universidad Cat\'olica de Valpara\'iso, Chile.}
\date{}
\maketitle

\abstract{
When using Neural Networks as trial functions to numerically solve PDEs, a key choice to be made is the loss function to be minimised, which should ideally correspond to a norm of the error. In multiple problems, this error norm coincides with--or is equivalent to--the $H^{-1}$-norm of the residual; however, it is often difficult to accurately compute it. This work assumes rectangular domains and proposes the use of a Discrete Sine/Cosine Transform to accurately and efficiently compute the $H^{-1}$ norm. The resulting Deep Fourier-based Residual (DFR) method efficiently and accurately approximate solutions to PDEs. This is particularly useful when solutions lack $H^{2}$ regularity and methods involving strong formulations of the PDE fail. We observe that the $H^1$-error is highly correlated with the discretised loss during training, which permits accurate error estimation via the loss.
}

\section{Introduction}
The use of Deep Learning techniques employing Neural Networks (NNs) have been sucessful to solve a wide range of data-based problems across fields such as image proccessing, healthcare, and autonomous cars \cite{afouras2018deep,alam2020survey,esteva2019guide, gupta2021deep,litjens2017survey,purushotham2018benchmarking,shorten2019survey, sluzalec2022quasi}. Recently, there has been a surge of interest in the use of neural networks as function spaces that can be employed to obtain numerical solutions of Partial Differential Equations (PDEs) \cite{berg2018unified, brevis2021machine, lu2021deepxde, paszynski2021deep, ruthotto2020deep, samaniego2020energy}. Owing to the universal approximation theorem, and variants in Sobolev spaces \cite{cybenko1989approximation,hornik1991approximation,hornik1990universal,kidger2020universal}, it is known that a sufficiently wide or deep NN is able to approximate any given continuous function on a compact domain with arbitrary accuracy, and thus they make suitable function spaces for solving PDEs. The use of automatic differentiation ({\it autodiff}) \cite{baydin2018automatic} facilitates efficient numerical evaluation of derivatives, which allows algorithmic differentiation of the neural network itself, as well as the use of gradient-based optimisation techniques such as Stochastic Gradient Descent (SGD) \cite{bottou2010large} and Adam \cite{kingma2015Adam} in order to minimise appropriate loss functions over a space of neural networks. 

A quantitative version of the Universal Approximation Theorem \cite{barron1993universal} demostrates that NNs can approximate without suffering the curse of dimensionality, requiring far fewer degrees of freedom to approximate functions with high-dimensional inputs than classical piecewise-linear function spaces, making them an attractive function space for solving PDEs, in particular, in high-dimensional problems. Beyond solving single instances of PDEs, NNs have shown a capacity to learn {\it operators} that solve families of parametrised PDEs, allowing rapid ``online" evaluation of solutions after an ``offline" training of the network \cite{chen1995universal,goswami2022physics,lagaris1998artificial, lu2019deeponet,lu2021learning}. 

The flexibility of NNs to solve many classes of PDEs is owed to a general and simple framework, whereby one chooses an appropriate architecture of the NN, a loss function, whose minimiser should be an exact solution of the PDE, and an optimisation procedure to attempt to minimise the loss function. In this article, we focus on the choice of loss function when solving PDEs with NN function spaces. Previous works have considered losses based on strong \cite{jagtap2020conservative,raissi2019physics,sirignano2018dgm} and weak \cite{khodayi2020deep,khodayi2020varnet,kharazmi2019variational} formulations of the PDE. However, the choice of a perfect loss function is generally not obvious as in practice solutions will only reach local minima, and the loss and error may have distinct or unknown convergence rates as one approaches either a local minimiser or a practically unattainable global minimiser.

Generally, a PDE operator can be described by a (possibly nonlinear) map $\mathcal{R}:X\to Y$, where $X,Y$ are Banach spaces. The PDE then takes the form
\begin{equation}\label{eqGenericPDE}
\mathcal{R}(u)=0.
\end{equation}
For example, Poisson's equation, $-\Delta u(x)=f(x)$, on a domain $\Omega$ with homogeneous Dirichlet boundary condition and $f\in L^2(\Omega)$ may be interpreted in strong form via the map\\ ${\mathcal{R}^s:H^2(\Omega)\cap H^1_0(\Omega)\to L^2(\Omega)}$ given by 
\begin{equation}
\mathcal{R}^s(u)=\Delta u+f,
\end{equation}
or in weak form via the map $\mathcal{R}^w:H^1_0(\Omega)\to H^{-1}(\Omega):=[H^1_0(\Omega)]^*$ given by 
\begin{equation}
\langle\mathcal{R}^w(u),v\rangle_{H^{-1}\times H^1}=\int_\Omega\nabla u(x)\cdot\nabla v(x)-f(x)v(x)\,dx.
\end{equation} 
When employing NNs to numerically solve PDEs, a loss is often selected as the norm of the PDE residual in $Y$, that is,
\begin{equation}
\mathcal{L}(u):=||\mathcal{R}(u)||_Y.
\end{equation}
One is generally confronted with two issues within this framework. The first is that norms on function spaces are generally given by integrals and thus a quadrature rule must be employed in order to numerically approximate $\mathcal{L}(u)$. In contrast to polynomial-based function spaces, an exact quadrature rule is generally unobtainable. Moreover, a poor choice of quadrature rule can lead to a form of ``overfitting" and poor approximation of solutions \cite{rivera2022quadrature}. The second issue is that in infinite dimensional spaces not all norms are equivalent, and thus the choice of norm on $Y$ can directly affect the convergence of the error during training. Ideally, we should employ norms on $X$ and $Y$ which are compatible in the sense that the $X$-norm of the error is equivalent to the $Y$-norm of the residual, leading to a residual minimisation method \cite{brevis2021machine,brevis2022neural,cier2021automatically}. 

Related to this second issue, progress has been made in the direction of {\it a priori} and {\it a posteriori} error estimates that allow estimation of the error via the loss  \cite{berrone2021variational,berrone2022solving, de2022error,mishra2021estimates,shin2020convergence,shin2020error}. The works rely on coercivity-type estimates of the error in terms of the exact norms, as well as a control of quadrature and training errors.

Many PDEs can be expressed in weak form via \eqref{eqGenericPDE} with $X=U$ is a space of trial functions, and $Y=V^*$, where $V$ is the space of test functions. That is,
\begin{equation}\label{eqWeakFormSomething}
\langle \mathcal{R}(u),v\rangle_{V^*\times V}=0\hspace{0.25cm}\forall v\in V.
\end{equation}
We commonly consider cases where $\mathcal{R}$ represents a linear and inhomogeneous PDE, and thus may be expressed in the form 
\begin{equation}\label{eqLinearPDE}
\langle\mathcal{R}(u),v\rangle_{V^*\times V}=b(u,v)-f(v),
\end{equation} where $f\in V^*$ and $b:U\times V\to\mathbb{R}$ is a bilinear form. It is clear that the PDE, in weak form, is equivalent to the statement that $||\mathcal{R}(u^*)||=0$ for any norm on $V^*$. The  most natural norm is the dual norm, induced by the norm on $V$, defined via 
\begin{equation}\label{eqDualNorm}
||f||_{V^*}=\sup\limits_{v\in V\setminus \{0\}}\frac{|f(v)|}{||v||_V}. 
\end{equation}

The advantage of employing the dual norm on $V^*$ is that, under certain assumptions that we will outline in more detail in \Cref{secDualNorms}, one can relate the dual norm of the residual to the norm of the error. Specifically, $\frac{1}{M}||\mathcal{R}(u)||_{V^*}\leq ||u-u^*||_U\leq \frac{1}{\gamma}||\mathcal{R}(u)||_{V^*}$, where $u$ is a candidate solution, $u^*$ is the exact solution, and $M,\gamma$ are positive, problem dependent, constants. This allows $||\mathcal{R}(u)||_{V^*}$ to be used as an error estimator, without needing to know the exact solution. In addition, if we can find a way to numerically approximate the dual norm, we can employ this as a loss function to be minimised over a trial function space.

We propose a Deep Fourier Residual (DFR) method to approximate the error of candidate solutions of PDEs in $H^1$ via an approximation of the dual norm of the residual of the PDE operator. The dual norm is then employed as a loss function to be minimised. The advantage of such a method is that the resulting norm is equivalent to the $H^1$-error of the solutions for certain well-posed problems.

We consider several numerical examples, comparing the DFR approach to other losses employed to solve differential equations using NNs. Our numerical examples exhibit strong correlation between the proposed loss and $H^1$-error during the training process. For sufficiently regular problems, our DFR method is qualitatively equivalent to existing methods in the literature (\Cref{ModelProblem1}) \cite{kharazmi2019variational, raissi2019physics}. However, in less regular problems, our method leads to significantly more accurate solutions, both for an equation that admits a smooth solution with large gradients (\Cref{ModelProblem2}), and for an elliptic equation with discontinuous parameters (\Cref{ModelProblem3}). Indeed, methods based on the strong formulation of the PDE, such as PINNs \cite{raissi2019physics}, cannot be implemented for such applications. The DFR method is shown to be advantageous both when solutions admit $H^1\setminus H^2$ regularity, and in regular problems where the forcing term has a large discrepency between its $L^2$ and $H^{-1}$ norm. We then consider further numerical experiments which demonstrate the DFR method's capability in a linear equation with point source (\Cref{ModelProblemDelta}), a nonlinear ODE (\Cref{ModelProblemNonlinear}), and a 2D linear problem (\Cref{ModelProblem2D}). 

The DFR method is currently limited to rectangular domains where each face has either a Dirichlet or a Neumann Boundary condition. We rely on a Fourier-type representation of the $H^{-1}$ norm that can be performed efficiently using the one-dimensional Discrete Cosine Transform and Discrete Sine Transform (DCT/DST), which are based on the Fast Fourier Transform (FFT), in each coordinate direction. Generally, an extension of our techniques to PDEs on arbitrary domains $\Omega$ would require access to an orthonormal basis of $H^1(\Omega)$, whose obtention may prove more costly than solving the PDE itself. Furthermore, the DST/DCT takes advantage of the FFT, which allows an inexpensive evaluation of the loss and would not be available in general domains. A possibility for the extension of the DFR method to arbitrary domains include methods analogous to embedded domain methods \cite{duster2008finite, glowinski2007distributed,larsson2022finite, mittal2005immersed, peskin2002immersed,ramiere2007fictitious, schillinger2015finite}, which embeds domains with complex geometry into a simpler fictious computational domain. It is also possible to borrow ideas from Goal-Oriented adaptivity (e.g., \cite{prudhomme1999goal}) to the proposed DFR method, although this will be postponed for a future work.

The structure of the paper is as follows. In \Cref{secPrelim} we cover some preliminary concepts. The theoretical groundwork for the definition of the DFR method is presented in \Cref{secTheory}, with our proposed loss defined in \Cref{subsecDefinitionLoss}. \Cref{secCompare} contains numerical examples comparing our proposed loss function with the VPINNs and collocation losses, which are roughly equivalent in regular problems, but we will demonstrate that the DFR method greatly outperforms VPINNs and PINNs when solutions are less regular. In \Cref{secExperiments} we consider further numerical experiments that demonstrate the DFR in equations with a point source, nonlinearities, and 2D results. Finally, concluding remarks are made in \Cref{secConc}. 

\section{Preliminaries}\label{secPrelim}

\subsection{Neural Networks}

Neural networks are functions expressed as  compositions of more elementary functions. In the simplest case of a fully connected feed-forward NN, an $M$-layer neural network is described by $M$ layer functions, $L_i:\mathbb{R}^{N_i}\to\mathbb{R}^{N_{i+1}}$, that are of the form 
\begin{equation}
L_i(x)=\sigma_i\left(A_ix+b_i\right),
\end{equation}
where $A_i$ is an $N_{i}\times N_{i+1}$ matrix, $b_i\in \mathbb{R}^{N_{i+1}}$, and $\sigma_i$ is an {\it activation function} that may depend on the layer index $i$ and acts component-wise on vectors. A fully-connected feed forward neural network is a function $\tilde{u}:\mathbb{R}^{N_1}\to\mathbb{R}^{N_{M+1}}$ defined by 
\begin{equation}\label{eqTildeU}
\tilde{u}(x)=L_{M}\circ L_{M-1}\circ\hdots \circ L_{1}(x).
\end{equation}
The final activation function $\sigma_M$ is taken to be the identity, $\sigma_M(x)=x$. The parameters $A_i,b_i$, known as the {\it weights} and {\it biases} of the network, parametrise the neural network. Optimisation over a neural network space with fixed architecture corresponds to identifying the optimal values of these trainable parameters. 

In the context of NNs for PDEs, we often need to impose homogeneous Dirichlet boundary conditions on our candidate solutions. In this work, we will do this by introducing a cutoff function. That is, if we wish to consider functions $u:\Omega\to\mathbb{R}$ so that for a subset of the boundary $\Gamma_D\subset\partial\Omega$, $u|_{\Gamma_D}=u_0$, we take $\tilde{u}$ to be of the form \eqref{eqTildeU}, and define 
\begin{equation}\label{eqArchit}
u(x)=\phi_1(x)\tilde{u}(x)+\phi_2(x),
\end{equation}
where $\phi_1$ is a function satisfying $\phi_1|_{\Gamma_D}=0$ and $\phi_1>0$ on $\bar{\Omega}\setminus\Gamma_D$, and $\phi_2|_{\Gamma_D}=u_0$. 

We include a schematic of this architecture in \Cref{figArchitecture}

\begin{figure}[H]

\begin{center}
\begin{tikzpicture}

\draw (0,2.85) node{Input};
\draw (0,2.5) node {$\overbrace{\hspace{1cm}}^{\hspace{1cm}}$} ;

\draw (5.5,4.1) node {Hidden layers};
\draw (5.5,3.75) node {$\overbrace{\hspace{8cm}}^\text{\hspace{1cm}}$} ;
%
%
%
%

\draw (4.5,-1.1) node {Trainable layers};
\draw (4.5,-0.75) node {$\underbrace{\hspace{10cm}}_\text{\hspace{1cm}}$} ;

\draw (9,2.85) node {$\tilde{u}$};
\draw (9,2.5) node {$\overbrace{\hspace{1cm}}^{\hspace{1cm}}$};

\draw(11,2.85) node {Output};

\draw (11,2.5) node {$\overbrace{\hspace{1cm}}^\text{\hspace{1cm}}$};
\Vertex[x=10,y=1.5,style={opacity=0.0,color=white!0},size=0.00000001]{inv1}
\Vertex[x=10,y=0,style={opacity=0.0,color=white!100}]{inv2}
\draw (10,0) node {Apply B.C.};
\Edge[Direct,lw=1pt](inv2)(inv1)

\Vertex[x=0,y=1.5,label=$x$,color=white,size=1,fontsize=\normalsize]{X}
\Vertex[x=2,y=3,color=white,size=0.75]{A11}
\Vertex[x=4,y=3,color=white,size=0.75]{A21}
\Vertex[x=7,y=3,color=white,size=0.75]{A31}
\Vertex[x=2,y=2,color=white,size=0.75]{A12}
\Vertex[x=4,y=2,color=white,size=0.75]{A22}
\Vertex[x=7,y=2,color=white,size=0.75]{A32}
\Vertex[x=2,y=0,color=white,size=0.75]{A13}
\Vertex[x=4,y=0,color=white,size=0.75]{A23}
\Vertex[x=7,y=0,color=white,size=0.75]{A33}

\Vertex[x=9,y=1.5,label=$\tilde{u}(x)$,color=white,size=1,fontsize=\normalsize]{tildeu}
\Vertex[x=11,y=1.5,label=$u(x)$,color=white,size=1,fontsize=\normalsize]{u}

\Vertex[x=5.5,y=3,style={color=white},label=$\color{black}\hdots$,size=0.75]{Ah1}
\Vertex[x=5.5,y=2,style={color=white},label=$\color{black}\hdots$,size=0.75]{Ah2}
\Vertex[x=5.5,y=0,style={color=white},label=$\color{black}\hdots$,size=0.75]{Ah3}

\Edge[Direct,lw=1pt](X)(A11)
\Edge[Direct,lw=1pt](X)(A12)
\Edge[Direct,lw=1pt](X)(A13)

\Edge[color=white,label={$\color{black}\vdots$}](A12)(A13)
\Edge[color=white,label={$\color{black}\vdots$}](A22)(A23)
\Edge[color=white,label={$\color{black}\vdots$}](Ah2)(Ah3)
\Edge[color=white,label={$\color{black}\vdots$}](A32)(A33)

\Edge[Direct,lw=1pt](A11)(A21)
\Edge[Direct,lw=1pt](A11)(A22)
\Edge[Direct,lw=1pt](A11)(A23)
\Edge[Direct,lw=1pt](A11)(A21)
\Edge[Direct,lw=1pt](A12)(A21)
\Edge[Direct,lw=1pt](A12)(A22)
\Edge[Direct,lw=1pt](A12)(A23)
\Edge[Direct,lw=1pt](A13)(A21)
\Edge[Direct,lw=1pt](A13)(A22)
\Edge[Direct,lw=1pt](A13)(A23)

\Edge[Direct,lw=1pt](A23)(Ah1)
\Edge[Direct,lw=1pt](A23)(Ah2)
\Edge[Direct,lw=1pt](A23)(Ah3)
\Edge[Direct,lw=1pt](A21)(Ah1)
\Edge[Direct,lw=1pt](A21)(Ah2)
\Edge[Direct,lw=1pt](A21)(Ah3)
\Edge[Direct,lw=1pt](A23)(Ah1)
\Edge[Direct,lw=1pt](A23)(Ah2)
\Edge[Direct,lw=1pt](A23)(Ah3)

\Edge[Direct,lw=1pt](Ah1)(A31)
\Edge[Direct,lw=1pt](Ah1)(A32)
\Edge[Direct,lw=1pt](Ah1)(A33)
\Edge[Direct,lw=1pt](Ah2)(A31)
\Edge[Direct,lw=1pt](Ah2)(A32)
\Edge[Direct,lw=1pt](Ah2)(A33)
\Edge[Direct,lw=1pt](Ah3)(A31)
\Edge[Direct,lw=1pt](Ah3)(A32)
\Edge[Direct,lw=1pt](Ah3)(A33)

\Edge[Direct,lw=1pt](A31)(tildeu)
\Edge[Direct,lw=1pt](A32)(tildeu)
\Edge[Direct,lw=1pt](A33)(tildeu)

\Edge[Direct,lw=1pt](tildeu)(u)

\end{tikzpicture}
\caption{NN architecture}
\label{figArchitecture}
\end{center}
\end{figure}
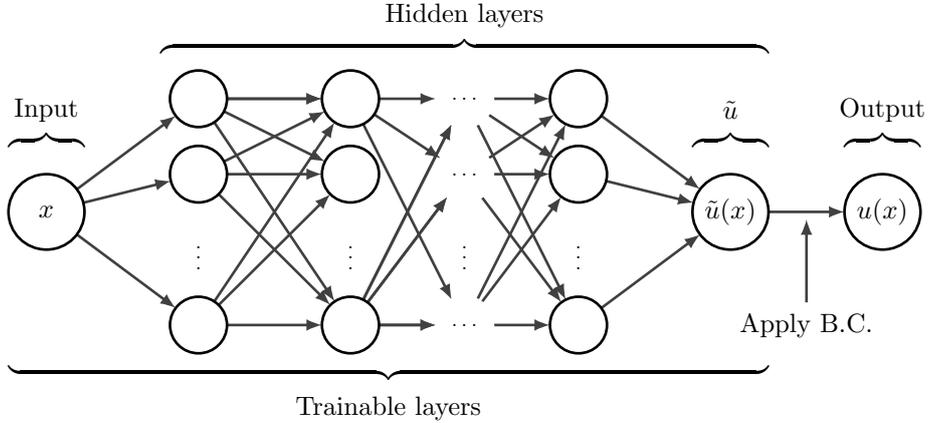

\subsection{PINN and VPINN losses}
Whilst there any many discrete losses employed when solving PDEs via NNs, in this section we outline two particular cases, the PINN (collocation), and the VPINN losses, which are based on strong and weak formulations of the PDE, respectively. These methods will be used for comparison in the numerical experiments of \Cref{secCompare}.

\subsubsection{Collocation}\label{subsecCol}

We assume that the strong form of the residual can be represented in the form 
\begin{equation}
\begin{split}
L_u(x)=&0 \, \hspace{0.2cm}(x\in \Omega),\\
G_u(x)=&0 \, \hspace{0.2cm}(x\in \partial\Omega).
\end{split}
\end{equation}
The collocation method considers discretisations of the $L^2$ norms of $L_u$ and $G_u$ as the loss function to be minimised, according to an appropriate quadrature rule. Explicitly, we consider the loss
\begin{equation}\label{eqColLoss}
\mathcal{L}_{col}(u):=\frac{1}{K_1}\sum\limits_{i=1}^{K_1} \omega_i|L_u(x_i)|^2+\frac{1}{K_2}\sum\limits_{i=1}^{K_2}\omega_i^b\left|G_u(x_i^b)\right|^2,
\end{equation}
where $(x_i)_{i=1}^{K_1}$ and $(\omega_i)_{i=1}^{K_1}$ are quadrature points in $\Omega$ and quadrature weights, respectively, which may be taken via a Monte Carlo or a deterministic quadrature scheme. Similarly, $(x_i^b)_{i=1}^{K_2}$ and $(\omega_i^b)_{i=1}^{K_2}$ are quadrature points and weights on the boundary.

In PDEs with low regularity, the strong form of the PDE does not hold and minimisers of Eq.  \eqref{eqColLoss} will not accurately represent the PDE. Despite this limitation, the collocation (PINN) method is one of the most attractive methods for regular problems as it is simple to implement using {\it autodiff} algorithms. Furthermore, by using Monte Carlo integration techniques, integrals can be estimated in high dimension without suffering from the curse of dimensionality.

\subsubsection{VPINNs}\label{subsecVPINN} 
VPINNs employ a loss that utilizes the weak formulation of the PDE. They correspond to a Petrov-Galerkin method where the trial space is given by NNs. Given a set of test functions $(v_k)_{k=1}^K$, a candidate solution $u$ and the residual $\mathcal{R}(u)\in V^*$ given in weak form, the loss is defined as 

\begin{equation}\label{eqDefL2Loss}
\mathcal{L}_{\vpinn}(u)=\sum\limits_{k=1}^K |\langle\mathcal{R}(u),v_k\rangle_{V^*\times V}|^2.
\end{equation}

In \cite{kharazmi2019variational}, this method was shown to be advantageous over classical PINNs method, both in terms of accuracy and speed. A particular application within their work, relevant to this manuscript, was to consider ODEs on $[0,1]$ with a NN architecture that consists of a single hidden layer with sine activation function, and test functions $v_k(x)=\sin(k \pi x)$. For this implementation, the authors were able to perform an exact quadrature to evaluate $\langle\mathcal{R}(u),v_k\rangle_{V^*\times V}$, which was employed in their loss function. In other implementations within their article, Legendre polynomials are considered as test functions. Whilst not directly commented within their work, in their implementation with sine test functions, the norm may be interpreted as a discretisation of the $L^2$-norm of the strong form of the residual. As they consider the test functions $(v_k)_{k=1}^K$ form to be a subset of an orthonormal basis of $L^2$, if there exists a strong form residual $L_u\in L^2$ such that 
$$\langle\mathcal{R}(u),v\rangle_{V^*\times V}=\langle L_u,v\rangle_{L^2}$$
for all $v\in H^1_0$, we observe that 
\begin{equation}
\sum\limits_{k=1}^K |\langle\mathcal{R}(u),v_k\rangle_{V^*\times V}|^2=\sum\limits_{k=1}^K\langle L_u,v_k\rangle_{L^2}^2\approx ||L_u||_{L^2}^2.
\end{equation}
In particular, for sufficiently regular problems, this implies that $\mathcal{L}_{\vpinn}$ and $\mathcal{L}_{col}$ each correspond to distinct discretisations of the same loss, i.e., the $L^2$-norm of the strong-form residual. The significant difference, however, is that the discretisation \eqref{eqDefL2Loss} is always well defined, even if the residual cannot be represented by an $L^2$ function, and we will observe the consequences of this distinction in \Cref{ModelProblem3}, employing sine-based test functions, as in \cite{kharazmi2019variational}.

\section{DFR Method} \label{secTheory}

For exposition purposes we will only list the key results necessary for defining the problem, with the details deferred to appendices. 
\subsection{Dual Norms and Residual Minimisation}\label{secDualNorms}

Let us consider a PDE of the form \eqref{eqWeakFormSomething}, described by a weak-form residual operator $\mathcal{R}:U\to V^*$, which is linear and inhomogeneous, so that it may be expressed as in \eqref{eqLinearPDE}. We assume $V$ to be a Hilbert space, and take $f\in V^*$, and $b:U\times V\to\mathbb{R}$ to be a bilinear form satsifying the continuity condition 
\begin{equation}\label{eqContLM}
|b(u,v)|\leq M||u||_U||v||_V
\end{equation}
and inf-sup stability condition 
\begin{equation}\label{eqBDDBelow}
\inf\limits_{u\in U\setminus \{0\}}\sup\limits_{v\in V\setminus \{0\}}\frac{|b(u,v)|}{||u||_U||v||_V}\geq \gamma,
\end{equation}
where $0<\gamma<M$. 

If we consider the map $B:U\to V^*$ given by $B:u\mapsto b(u,\cdot)$, then the conditions \eqref{eqContLM} and \eqref{eqBDDBelow} ensure that $B$ is a boundedly invertible map onto its image $B(U)$, and, in particular, if $B$ is surjective, then there exists a unique solution to \eqref{eqWeakFormSomething} for all $f\in V^*$ \cite[Section 6.12]{ciarlet2013linear}. Furthermore, if $f$ is in the range of $B$, given the exact solution $u^*\in U$ to \eqref{eqWeakFormSomething}, and a candidate solution $u\in U$, we may estimate the error using the dual norm of the residual via the inequalities
\begin{equation}\label{eqUpLowerBd}
\frac{1}{M}||\mathcal{R}(u)||_{V^*}\leq ||u-u^*||_U\leq \frac{1}{\gamma}||\mathcal{R}(u)||_{V^*}.
\end{equation}
Both inequalities are found by noting that since $\mathcal{R}(u^*)=0$, then $$\langle\mathcal{R}(u),v\rangle_{V^*\times V}=\langle\mathcal{R}(u),v
\rangle_{V^*\times V}-\langle\mathcal{R}(u^*),v\rangle_{V^*\times V}=b(u-u^*,v)$$ for any test function $v\in V$. Correspondingly, the lower bound is a direct consequence of \eqref{eqContLM}, as 
\begin{equation}
\sup\limits_{v\in V\setminus\{0\}}\frac{|\langle\mathcal{R}(u),v\rangle_{V^*\times V}|}{||v||_V}=\sup\limits_{v\in V\setminus \{0\}}\frac{|b(u-u^*,v)|}{||v||_V}\leq M||u-u^*||_{U}.
\end{equation}
Similarly, the upper bound is a direct consequence of \eqref{eqBDDBelow}, as 
\begin{equation}
\begin{split}
\sup\limits_{v\in V\setminus \{0\}} \frac{|\langle\mathcal{R}(u),v\rangle_{V^*\times V}|}{||v||_V}=\sup\limits_{v\in V\setminus \{0\}}\frac{|b(u-u^*,v)|}{||v||_V}\geq \gamma||u-u^*||. 
\end{split}
\end{equation} 
This makes $||\mathcal{R}(u)||_{V^*}$ a natural choice of norm to be utilised as a loss function for training NNs. Generally, however, it is non-trivial to evaluate $||\cdot||_{V^*}$, which is defined as in \eqref{eqDualNorm}, and, unlike classical Sobolev-type norms, generally cannot be expressed as a single integral of the function and its derivatives. 


\subsubsection{Nonlinear equations}\label{remarkNonlinearity}

Whilst our previous discussion applies only to linear equations, via a linearisation argument it is possible to obtain a local version of \eqref{eqUpLowerBd} for nonlinear PDE. We consider an abstract PDE given of the following form. Considering $\mathcal{R}:U\to V^*$ to be nonlinear, we obtain the following result. 

\begin{proposition}\label{propNonlinearEstimate}
Assume that there exists a (possibly non-unique) solution $u^*\in U$ of ${\mathcal{R}(u^*)=0}$ such that:
\begin{enumerate}[(i)]
\item There exists $r_0>0$ such that $\mathcal{R}$ is Gateaux differentiable for all $u\in U$ with $||u-u^*||_{U}<r_0$.
\item The directional derivative $\delta_w\mathcal{R}(u^*)$ is bounded below, so that there exists $\gamma>0$ such that for all $w\in U$
$$||\delta_w\mathcal{R}(u^*)||_{V^*}\geq \gamma||w||_U.$$
\item There exists some $r_0>0$ such that the Gateaux derivative of $\mathcal{R}$ is Lipschitz on the ball $B_{r_0}(u^*)$. 
\end{enumerate}
Then for every $0<\epsilon<1$ there exists $\delta>0$ such that if $||u-u^*||_U<\delta$, 
\begin{equation}
\frac{1-\epsilon}{M}||\mathcal{R}(u)||_{V^*}\leq ||u-u^*||_U\leq \frac{1+\epsilon}{\gamma}||\mathcal{R}(u)||_{V^*}. 
\end{equation}
\end{proposition}

We defer the precise definitions of the objects in the proposition and its proof to \Cref{appNonlinear}. If $\mathcal{R}$ corresponds to a linear inhomogeneous PDE and is thus equal to its linearisation, $\delta$ can be taken as $+\infty$ and the estimate is global. The significance of this result is that if we have a PDE described by a sufficiently regular $\mathcal{R}$ and a candidate solution sufficiently close to an exact solution, which need not be unique, they will satisfy estimates analogous to \eqref{eqUpLowerBd}. In particular, if $||\mathcal{R}(u)||_{V^*}$ is used as a loss function, and the candidate solution is close enough to the exact one (which is to be expected at the end of trianing), we should observe strong a correlation between the loss and the $H^1$-error. We will numerically illustrate this in \Cref{ModelProblemNonlinear}. 

\subsection{Evaluation of the $H^{-1}$ norm with the DFR method}

This work takes advantage of a Parseval-type inequality to evaluate the dual norm of elements of the dual space in terms of an orthonormal basis of the original Hilbert space.

\begin{proposition}
Let $V$ be a real separable Hilbert space with inner product ${\langle\cdot,\cdot\rangle_V:V\times V\to\mathbb{R}}$, and $(\varphi_k)_{k=1}^\infty$ a countable, orthogonal basis of $V$, with $||\varphi_k||_{V}^2=:\lambda_k$. Then for all $f\in V^*$, 
\begin{equation}\label{eqDualNormSeries}
||f||_{V^*}:=\max\limits_{v\in V\setminus \{0\}}\frac{|f(v)|}{||v||_V}=\left(\sum\limits_{k=1}^\infty \lambda_k^{-1}f(\varphi_k)^2\right)^\frac{1}{2}.
\end{equation}

\end{proposition}
\begin{proof}
By the Riesz Representation Theorem, there exists a unique solution $u_f\in V$ of
\begin{equation}\label{eqRRTint}
\langle u_f,v\rangle_{V}=f(v)
\end{equation}
for all $v\in V$, that satisfies $||u_f||_{V}=||f||_{V^*}$, and thus the mapping $f\mapsto u_f$ is an isometry. As $(\varphi_k)_{k=1}^\infty$ is an orthogonal basis of $V$, we then have via the generalised Parseval identity that 
\begin{equation}
||f||_{V^*}=||u_f||_{V}=\left(\sum\limits_{k=1}^\infty \frac{\langle u_f,\varphi_k\rangle_{V}^2}{||\varphi_k||_{V}^2}\right)^\frac{1}{2}=\left(\sum\limits_{k=1}^\infty \lambda_k^{-1}f(\varphi_k)^2\right)^\frac{1}{2}.
\end{equation}
\end{proof}

Our approach consists of approximating the dual norm of elements $f\in V^*$ using a truncated version of this series expression. To do so, we need an appropriate orthogonal basis $(\varphi_k)_{k=1}^\infty$ of $V$. 

In this work, we restrict ourselves to problems where the space of test functions is $V=\{u\in H^1(\Omega): u|_{\Gamma_D}=0\}$, where $\Gamma_D\subset\partial\Omega$ is the region of the boundary corresponding to a Dirichlet boundary condition of the PDE. For this case, we take our orthogonal basis of $V$ to be the (weak) solutions of 

\begin{equation}
\begin{split}
\lambda_k\int_\Omega v(x)\varphi_k(x)=&\int_\Omega \nabla \varphi_k(x)\cdot \nabla v(x)+\varphi_k(x)v(x)\,dx \hspace{0.5cm}(\forall v\in V)\\
||\varphi_k||_{L^2}=&1.
\end{split}
\end{equation}
In strong form, we may write the PDE as 
\begin{equation}
(1-\Delta)\varphi_k(x)=\lambda_k\varphi_k(x).
\end{equation}
That is, $\varphi_k$ are eigenvectors of $1-\Delta$ with homogeneous Dirichlet condition on $\Gamma_D$ and homogeneous Neumann condition on $\partial\Omega\setminus\Gamma_D=\Gamma_N$, with corresponding eigenvalues $\lambda_k$, and normalised to have unit norm in $L^2(\Omega)$. The key properties of $\varphi_k$, with proofs deferred to \Cref{appLaplacian} and derived from classical spectral theory, are: 
\begin{enumerate}
\item $(\varphi_k)_{k=1}^\infty$ forms an orthogonal basis of $V$, and an orthonormal basis of $L^2(\Omega)$.
\item $||\varphi_k||_{H^1}^2=\lambda_k$, where $\lambda_k$ is the eigenvalue of $1-\Delta$ corresponding to $\varphi_k$. 
\item $\lambda_k\geq 1$ for all $k$ and, under a suitable reordering, $\lambda_k$ is non-decreasing and unbounded.
\end{enumerate}

In this case, we see that $f$ and $u_f$ are related via $(1-\Delta)u_f=f$, and we may interpret \eqref{eqDualNormSeries} as stating that $||f||_{V^*}=||(1-\Delta)^{-1}f||_{H^1}$, where our series expression allows us to evaluate the latter in a straightforward manner.

In general geometries, it is non-trivial to identify the eigenvectors and eigenvalues of $1-\Delta$ on a domain, and thus for the sake of this work we consider only simple geometries described by $n$-dimensional cubes, $\Omega=(0,\pi)^n$, with $\Gamma_D$ given by a union of any number of the $2n$ faces of $\partial\Omega$. In this case, the eigenvectors of $1-\Delta$ are simply products of the eigenvectors of $1-\frac{d}{dx_i^2}$ in each coordinate direction with the appropriate boundary conditions, and thus may be written explicitly. Moreso, as these eigenvectors are all described as sines and cosines of varying frequencies, we will be able to take advantage of the Discrete Sine/Cosine Transforms to efficiently evaluate the residual at the basis functions. 

Considering the one-dimensional problem, as $\partial[0,\pi]$ consists of only two points, there are only four options for $\Gamma_D\subset \partial[0,\pi]$. \Cref{tableBases} lists the four possible boundary conditions, along with corresponding eigenvalues $\lambda_k$ and eigenvectors $\varphi_k$.
\begin{table}[h!]
\begin{center}
\begin{tabular}{|m{1cm} |m{2cm}| m{4.5cm}| m{1.5cm}|m{1.5cm}|@{}m{0pt}@{}}
\hline
$\Gamma_D$& $\lambda_k$ & $ \varphi_k$ & $\langle \cdot,\varphi_k\rangle_{L^2}$ &$\langle \cdot,\varphi_k'\rangle_{L^2}$&\\[10pt]
\hline
$\{0,\pi\}$&  $1+k^2$ & $\sqrt{\frac{2}{\pi}}\sin(kx)$ & DST-II & DCT-II&\\[10pt]
\hline
$\{0\}$ & $1+\left(k-\frac{1}{2}\right)^2$ & $\sqrt{\frac{2}{\pi}}\sin\left(\left(k-\frac{1}{2}\right)x\right)$ & DST-IV & DCT-IV&\\[10pt]
\hline
$\{\pi\}$ & $1+\left(k-\frac{1}{2}\right)^2$ & $\sqrt{\frac{2}{\pi}}\cos\left(\left(k-\frac{1}{2}\right)x\right)$& DCT-IV & DST-IV&\\[10pt]
\hline
$\emptyset$ & $1+(1-k)^2$ & $\sqrt{\frac{2}{\pi}}\cos\left((k-1)x\right)-\delta_{k1}\pi^{-\frac{1}{2}}$ &DCT-II & DST-II&\\[10pt]
\hline
\end{tabular}
\end{center}
\caption{Basis functions and eigenvalues for $H^1(0,\pi)$ with various boundary conditions, along with the relevant transforms for evaluating integrals against the basis functions and their derivatives.}\label{tableBases}
\end{table}

To evaluate PDE residuals acting on basis functions of the above forms, we will need to numerically evaluate integrals involving basis functions and their derivatives. As these are global basis functions, a naive calculation of the integrals of each basis function could prove prohibitively expensive. To remedy this, we consider the Discrete Sine/Cosine transforms as a means of quadrature, which reduce the number of calculations required to use an $N$ point midpoint rule for $N$ basis functions from $O(N^2)$ to $O(N\ln N)$. \Cref{tableBases} contains in its fourth and fifth columns the quadrature scheme for evaluating integrals against the basis functions and their derivatives, respectively. We defer detailed discussion of the transforms to \Cref{appFFT}. Their key use is that they are analogous to the Fast Fourier Transform, where the boundary conditions are no longer periodic. In fact, their efficient calculation arises from their representation as special cases of the Discrete Fourier Transform under particular symmetries. 

These basis functions are easily adapted to arbitrary intervals $(a,b)$ by a rescaling argument. We may define corresponding orthonormal basis functions $\tilde{\varphi}_k\in H^1(a,b)$ by 
\begin{equation}
\tilde{\varphi}_k(x)=\varphi_k\left(\frac{\pi(x-a)}{b-a}\right),\end{equation}
which are eigenvectors of $1-\Delta$ on $(a,b)$ with corresponding eigenvalues
\begin{equation}
\tilde{\lambda}_k=\frac{\pi^2}{(b-a)^2}(\lambda_k-1)+1.
\end{equation}

Similarly, by considering tensor products of these 1D basis functions, we can obtain an orthonormal basis for $H^1(\Omega)$ with the appropriate boundary conditions when $\Omega=\prod\limits_{i=1}^d(a_i,b_i)$.

\subsection{Definition of the discretised loss}\label{subsecDefinitionLoss}

Our aim is to define a computable, discretised loss, $\mathcal{L}_{V^*}$, such that for a candidate solution $u$, we have that 
\begin{equation}\label{eqLossApprox}
\mathcal{L}_{V^*}(u)\approx ||\mathcal{R}(u)||^2_{V^*}.
\end{equation}
We will do this by employing a truncated series expansion of \eqref{eqDualNormSeries}, and taking the sine/cosine based basis functions outlined in \Cref{tableBases}, with quadrature corresponding to the DCT/DST, as outlined in \Cref{appFFT}. Before defining the loss in the general case, we outline its definition in a simple one-dimensional example for clearer exposition. 

\subsubsection{One-dimensional example}\label{subsubsecExample}
Let $f\in L^2(0,\pi)$ and $g\in\mathbb{R}$. Let us consider the ODE, in weak form, to be: find $u \in H^1(0,\pi)$ with $u(0)=0$, such that 
\begin{equation}
\langle\mathcal{R}(u),v\rangle_{V^*\times V}=\int_0^\pi \sigma(x)u'(x)v'(x)+f(x)v(x)\,dx -gv(\pi)=0
\end{equation}
for all $v\in H^1(0,\pi)$ with $v(0)=0$. Consulting \Cref{tableBases}, we see that our relevant basis functions for $V$ corresponding to our homogeneous Dirichlet boundary condition at $0$ are given by  ${\varphi_k(x)=\sqrt{\frac{2}{\pi}}\sin\left(\left(k-\frac{1}{2}\right)x\right)}$ with corresponding eigenvalues $\lambda_k=1+\left(k-\frac{1}{2}\right)^2$. The derivatives of our basis functions are readily evaluated as $\varphi_k'(x)=\left(k-\frac{1}{2}\right)\sqrt{\frac{2}{\pi}}\cos\left(\left(k-\frac{1}{2}\right)x\right)$. 

First, we choose a truncation frequency, $N>0$. For $1\leq k\leq N-1$, we now aim to approximate $\langle\mathcal{R},u,\varphi_k\rangle_{V^*\times V}$. Recalling the approximations in \eqref{eqFFTs}, we may then approximate the integrals appearing in the residual as 
\begin{equation}
\begin{split}
\int_0^\pi f(x)\varphi_k(x)\,dx =&  \int_0^\pi f(x)\sqrt{\frac{2}{\pi}}\sin\left(\left(k-\frac{1}{2}\right)x\right)\,dx\\
\approx & \mathcal{S}_{N,k}^{II}(f),\\
\int_0^\pi \sigma(x)u'(x)\varphi_k'(x)\,dx = & \int_0^\pi\left(k-\frac{1}{2}\right)\sqrt{\frac{2}{\pi}}\cos\left(\left(k-\frac{1}{2}\right)x\right)\sigma(x)u'(x)\,dx\\
\approx & \left(k-\frac{1}{2}\right)\mathcal{C}_{N,k}^{II}(\sigma u'),
\end{split}
\end{equation}
where {\it autodiff} is employed to evaluate $u'$ for a candidate solution $u$ described by a NN, and $\mathcal{C}^{II}_{N,k},\mathcal{S}_{N,k}^{II}$ are the type-II DCT and DST, respectively, as described in \Cref{appFFT}. The boundary term may be evaluated exactly as 
\begin{equation}
g\varphi_k(\pi)=g\sqrt{\frac{2}{\pi}}\sin\left(\left(k-\frac{1}{2}\right)\pi\right)=g\sqrt{\frac{2}{\pi}}(-1)^{k+1}.
\end{equation}
Thus, we define the discretised transform of the residual, $\hat{\mathcal{R}}(u)(k)$ for $1\leq k \leq N-1$ as 
\begin{equation}
\hat{\mathcal{R}}(u)(k):=\left(k-\frac{1}{2}\right)\mathcal{C}_{N,k}^{II}(\sigma u')+\mathcal{S}_{N,k}^{II}(f)-g\sqrt{\frac{2}{\pi}}(-1)^{k+1}.
\end{equation}
Finally, our discretised loss, denoted $\mathcal{L}_{V^*}$, is defined to be 
\begin{equation}
\mathcal{L}_{V^*}(u):=\sum\limits_{k=1}^{N-1}\frac{|\hat{\mathcal{R}}(u)(k)|^2}{1+\left(k-\frac{1}{2}\right)^2}.
\end{equation}

\subsubsection{The general case}
\label{subsubsecLoss}

Let $\Omega=(0,\pi)^d$, and take $\Gamma_D\subset \partial\Omega$ to be a union of faces of the rectangular domain $\Omega$. Take $V=\{u\in H^1(\Omega):u|_{\Gamma_D}=0\}$. For a candidate solution $u\in H^1(\Omega)$, consider the residual 
\begin{equation}
\langle\mathcal{R}(u),v\rangle_{V^*\times V}=\int_\Omega F^1_u(x)\cdot \nabla v(x)+F^2_u(x)v(x)\,dx -\int_{\Gamma_N}G_u(x)v(x)\,dx,
\end{equation}
where the functions $F^1_u,F^2_u,G_u$ may be functions of $x,u$ and derivatives of $u$. 

The loss is thus defined according to the following process:

\begin{enumerate}
\item Choose a cutoff frequency $N>0$. 
\item Identify the correct basis functions $\varphi_{k_1,k_2,...,k_d}(x_1,x_2,...,x_d)=\prod\limits_{i=1}^d\varphi^i_{k_i}(x_i)$, where $\varphi^i_{k_i}$ are the 1D basis functions described in \Cref{tableBases} according to the boundary conditions on each face. 
\item Identify the correct eigenvalues $\lambda_{k_1,k_2,...,k_d}=1-d+\sum\limits_{i=1}^d \lambda_{k_i}$ for each basis function according to \Cref{tableBases}. 
\item Express the residual operator $\mathcal{R}(u)$ in weak form, evaluate integrals across the interior and faces by performing the appropriate DCT/DST in each coordinate direction, according to the fourth and fifth columns of \Cref{tableBases} for $k_1,...,k_d=1,...,N-1$ to give $\langle\mathcal{R}(u),\varphi_{k_1,...,k_d}\rangle_{V^*\times V}\approx \hat{\mathcal{R}}(u)(k_1,...,k_d)$. 
\item Evaluate the loss as 
$$\mathcal{L}_{V^*}(u):=\sum\limits_{i=1}^d\sum\limits_{k_i=1}^{N-1}\frac{|\hat{\mathcal{R}}(u)(k_1,k_2,...,k_d)|^2}{\lambda_{k_1,k_2,...,k_d}}.$$
\end{enumerate}

\subsection{Potential limitations}
We have two sources of error in the approximation \eqref{eqLossApprox}. First, the error arising from quadrature, according to our mid-point rule for integration and evaluation of $\hat{\mathcal{R}}(u)(k)$. Second, errors arise from the truncation of the infinite series in \eqref{eqDualNormSeries}. In contrast, the quadrature rule employed by the DCT/DST is exact when $\mathcal{R}(u)$, represented as a function in $H^1$ via the Riesz Representation Theorem belongs to the span of $\left((\varphi_{k_1k_2,...,k_d})_{k_i=1}^N\right)_{i=1}^d$, where $N$ is the cutoff frequency employed. That is, our discretisation error corresponds to high-frequencies of the {\it residual}--not to be confused with high frequencies of the solution. In principle, high-frequency components in the residual may arise as a consequence of the following:
\begin{enumerate}
\item The residual itself contains high-frequency modes due to the presence of terms with low-regularity.
\item The function space used has low regularity, such as NNs with a $ReLU$ activation function. 
\item The function space is flexible enough and the number of integration points is low enough that overfitting occurs during minimisation, which would introduce high-frequency modes, unseen by the loss. 
\end{enumerate}

By using a sufficiently high cutoff frequency $N>0$, errors corresponding to (1) should be negligible. To avoid the possibility of (2), we employ smooth activation functions. With regards to (3), it is known that extreme quadrature issues can arise when training NNs to solve PDEs \cite{rivera2022quadrature}, which may be interpreted as a form of overfitting. We do not focus on this issue within this work and we use a validation and a training set to verify if overfitting occurs.

\section{Numerical Experiments}
\subsection{Validation Results}\label{secCompare}

We now consider various linear ODEs to illustrate the differences between different losses for training a neural network when solving linear PDEs. $\mathcal{L}_{V^*},\mathcal{L}_{\vpinn}$, and $\mathcal{L}_{col}$ denote our proposed method, the VPINNs loss, and the collocation method, respectively. The obtained solutions are denoted as $u_{V^*},u_{\vpinn}$, and $u_{col}$, respectively. 

\Cref{figArchitecture} describes the NN architecture of our candidate solutions. It consists of  five hidden layers with $\tanh$ activation function and 25 neurons per layer. We only consider homogeneous Dirichlet boundary conditions, which are implemented according to \eqref{eqArchit} by taking $\phi_1(x)=x$ when $\Gamma_D=\{0\}$, and $\phi_1(x)=x(\pi-x)$ when $\Gamma_D=\{0,\pi\}$. For consistency between experiments, each candidate solution is initialised with the same weights and biases. 

Our implementation uses {\it Tensorflow 2.8}. We use Adam as our optimiser with initial learning rate $10^{-2}$, and an adaptive learning rate, as defined in \cite{uriarte2022finite}, and implemented via a {\it Callback}. The adaptive learning rate allows the optimiser to select the ``correct" learning rate according to the decay of the loss, and rejects iteration steps which lead to an increase in the loss. This choice accelerates convergence, and allows a fairer comparison between the three methods considered, as otherwise the convergence may be highly dependent on the selected learning rate. 

In each case, we minimise the loss using $200$ points, which in the Fourier-based losses corresponds to employing the first 200 basis functions in the truncated series expansions \eqref{eqDualNormSeries} and \eqref{eqDefL2Loss}, and 200 equispaced integration points in the collocation method with a mid-point integration rule. We also measure the loss on a validation set of $274$ points so that we may see if overfitting takes place, which does not factor into the updating of the NN weights, and is only used as a metric for comparison after training. In the Fourier-based losses, this also corresponds to a total of $274$ frequencies used to evaluate the validation loss. 

\subsubsection{Losses implemented}

For comparison, we implement the collocation based loss as described in \Cref{subsecCol}, and the VPINNs loss as described in \Cref{subsecVPINN}. For the latter, we consider an implementation that is highly comparable to the loss $\mathcal{L}_{V^*}$ that we propose. Explicitly, we take 

\begin{equation}\label{eqL2Used}
\mathcal{L}_{\vpinn}(u):=\sum\limits_{k=1}^N|\hat{\mathcal{R}}(u)(k)|^2.
\end{equation}
where $\hat{\mathcal{R}}(u)(k)$ is defined in \Cref{subsubsecLoss}, and we evaluate this using DST/DCT. Thus, we have an application of VPINNs that allows a direct comparison between using an $L^2$-based norm and an $H^{-1}$-based norm for evaluating the PDE residual in weak form. When $V=H^1_0(0,\pi),$ this is consistent with VPINNs as considered in \cite[Section 4.1]{kharazmi2019variational}, which used sine-based test functions albeit with a different quadrature rule. The only difference between \eqref{eqL2Used} and $\mathcal{L}_{V^*}$ is the weighting factor $\lambda_k^{-1}$ imposed in the summation. As discussed in \Cref{subsecVPINN}, this leads to the interpretion of $\mathcal{L}_{\vpinn}$ as a discretisation of the $L^2$ norm of the strong formulation of the residual. As $\mathcal{L}_{col}$ is also a discretisation of the $L^2$ norm of the strong-form residual, we expect implementations that utilise $\mathcal{L}_{col}$ and $\mathcal{L}_{\vpinn}$ to generally behave the same when the strong form is well-defined. The key difference, however, is that $\mathcal{L}_{\vpinn}$ is still well-defined when the strong form of the PDE is not equivalent to the weak form.

\subsubsection{Model Problem 1 - Smooth solution}
\label{ModelProblem1}

We consider the following ODE in variational form: find $u\in H^1_0(0,\pi)$ satisfying
\begin{equation}
\int_0^\pi u'(x) v'(x)-4\sin(2x)v(x)\,dx =0
\end{equation}
for all $v\in H^1_0(0,\pi)$. This has exact solution given by $u^*(x)=\sin(2x)$. 

This example is selected because the weak and strong forms of the PDE are equivalent, and the solution only admits low-frequency modes. For such problems, we expect that $\mathcal{L}_{\vpinn}(u)\approx \mathcal{L}_{col}(u)$, with the approximation being exact in the limit as the number of sampling points tends to infinity. We further expect all three implemented losses to behave similarly since only low-frequency modes play a significant role in this problem.

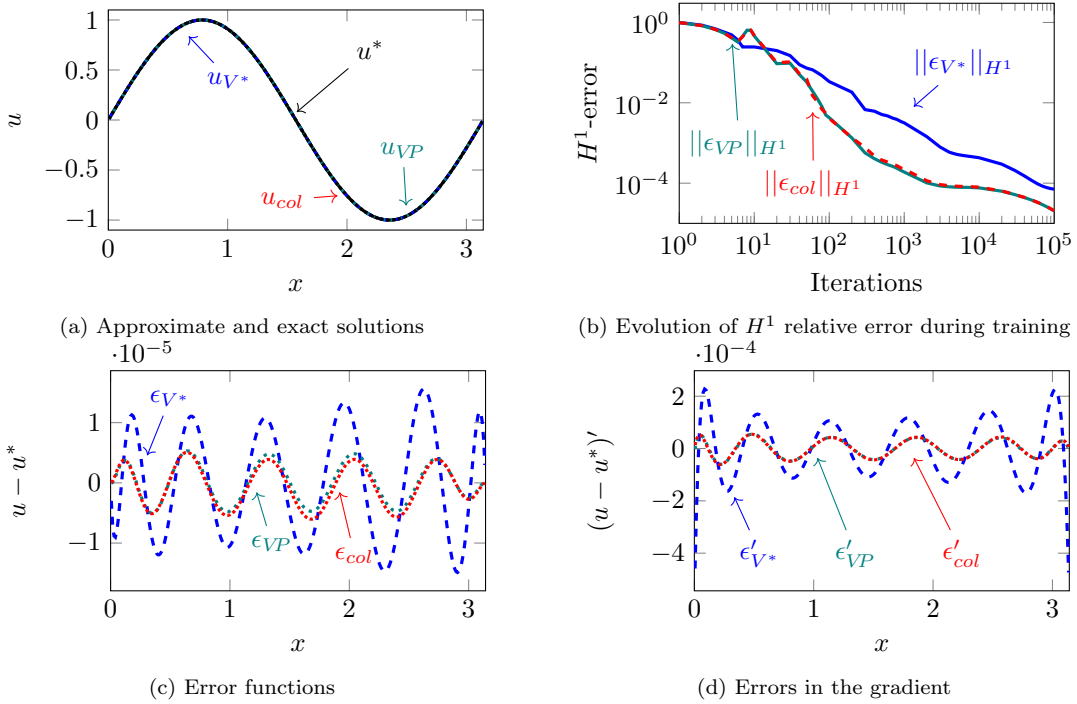
\begin{figure}[H]
\begin{tabular}{c c}
\begin{subfigure}[t]{0.475\textwidth}
\begin{center}
\begin{tikzpicture}
\begin{axis}[
  xlabel=$x$,
  ylabel=$u$,
xmin = 0,
xmax=3.1416,
ymin=-1.1,
ymax=1.1,
height=4.5cm,
width=0.9\textwidth]
\addplot [mark=none,very thick] table [y=b, x=a, col sep = comma]{Graphics/Data/Sine/SolExact.csv};
\node[anchor=west] (source) at (axis cs:2.0,0.7){$u^*$};
       \node (destination) at (axis cs:1.5,0){};
       \draw[->](source)--(destination);
\addplot [mark=none, blue, very thick,dashed ]table [y=b, x=a, col sep = comma]{Graphics/Data/Sine/SolDual.csv};
\node[anchor=west] (source) at (axis cs:0.75,0.4){\color{blue}$u_{V^*}$};
       \node (destination) at (axis cs:0.6,0.93){};
       \draw[blue, ->](source)--(destination);
\addplot [mark=none, teal!100, very thick,dotted] table [y=b, x=a, col sep = comma]{Graphics/Data/Sine/SolL2.csv};
\node[anchor=west] (source) at (axis cs:2.2,-0.3){\color{ teal!100}$u_{\vpinn}$};
       \node (destination) at (axis cs:2.5,-0.95){};
       \draw[teal, ->](source)--(destination);
\addplot [mark=none red, very thick, densely dotted] table [y=b, x=a, col sep = comma,mark=none]{Graphics/Data/Sine/SolCol.csv};
\node[anchor=west] (source) at (axis cs:1.2,-0.8){\color{red}$u_{col}$};
       \node (destination) at (axis cs:2,-0.75){};
       \draw[red, ->](source)--(destination);
\end{axis}
\end{tikzpicture}
\caption{ Approximate and exact solutions}
\end{center}
\end{subfigure}&
\begin{subfigure}[t]{0.475\textwidth}
\begin{center}
\begin{tikzpicture}
\begin{axis}[
  xlabel=Iterations,
  ylabel=$H^1$-error,
xmin =1,
xmax=10e4,
ymin=10e-6,
ymax=3,
height=4.5cm,
width=0.9\textwidth,
xmode=log,
ymode=log]
\addplot [mark=none, blue, very thick ]table [y expr=sqrt(\thisrow{c}),x=a, col sep = comma]{Graphics/Data/Sine/LossH1Dual.csv};
\node[anchor=west] (source) at (axis cs:10e2,0.1){\color{blue}$||\epsilon_{V^*}||_{H^1}$};
       \node (destination) at (axis cs:10e2,5*10e-4){};
       \draw[blue, ->](source)--(destination);
\addplot [mark=none,teal,very thick ]table [y expr=sqrt(\thisrow{c}),x=a, col sep = comma]{Graphics/Data/Sine/LossH1L2.csv};
\node[anchor=west] (source) at (axis cs:1,10e-4){\color{teal}$||\epsilon_{\vpinn}||_{H^1}$};
       \node (destination) at (axis cs:5,5*10e-2){};
       \draw[teal, ->](source)--(destination);
\addplot [mark=none, red, very thick ,dashed]table [y expr=sqrt(\thisrow{c}),x=a, col sep = comma]{Graphics/Data/Sine/LossH1Col.csv};
\node[anchor=west] (source) at (axis cs:10,10e-5){\color{red}$||\epsilon_{col}||_{H^1}$};
       \node (destination) at (axis cs:60,10e-3){};
       \draw[red, ->](source)--(destination);
\end{axis}
\end{tikzpicture}
\caption{Evolution of $H^1$ relative error during training}
\end{center}
\end{subfigure}\\
\begin{subfigure}[t]{0.475\textwidth}
\begin{center}
\begin{tikzpicture}
\begin{axis}[
  xlabel=$x$,
  ylabel=$u-u^*$,
xmin = 0,
xmax=3.1416,
height=4.5cm,
width=0.9\textwidth]
\addplot [mark=none, blue, very thick,dashed ]table [y=b, x=a, col sep = comma]{Graphics/Data/Sine/ErDual.csv};
\node[anchor=west] (source) at (axis cs:0.25,1.4*10e-6){\color{blue}$\epsilon_{V^*}$};
       \node (destination) at (axis cs:0.275,0.3*10e-6){};
       \draw[blue, ->](source)--(destination);
\addplot [mark=none, teal, very thick,dotted] table [y=b, x=a, col sep = comma]{Graphics/Data/Sine/ErL2.csv};
\node[anchor=west] (source) at (axis cs:1.1,-1*10e-6){\color{teal}$\epsilon_{\vpinn}$};
       \node (destination) at (axis cs:1.2,0){};
       \draw[teal, ->](source)--(destination);
\addplot [mark=none, red, very thick, densely dotted] table [y=b, x=a, col sep = comma,mark=none]{Graphics/Data/Sine/ErCol.csv};
\node[anchor=west] (source) at (axis cs:1.8,-1.2*10e-6){\color{red}$\epsilon_{col}$};
       \node (destination) at (axis cs:1.9,0){};
       \draw[red, ->](source)--(destination);
\end{axis}
\end{tikzpicture}
\caption{Error functions}
\end{center}
\end{subfigure}&
\begin{subfigure}[t]{0.475\textwidth}
\begin{center}
\begin{tikzpicture}
\begin{axis}[
  xlabel=$x$,
  ylabel=$(u-u^*)'$,
xmin = 0,
xmax=3.1416,
height=4.5cm,
width=0.9\textwidth]
\addplot [mark=none, blue, very thick,dashed ]table [y=b, x=a, col sep = comma]{Graphics/Data/Sine/GradErDual.csv};
\node[anchor=west] (source) at (axis cs:0.3,-4*10e-5){\color{blue}$\epsilon_{V^*}'$};
       \node (destination) at (axis cs:0.3,-1.5*10e-5){};
       \draw[blue, ->](source)--(destination);
\addplot [mark=none, teal, very thick,dotted] table [y=b, x=a, col sep = comma]{Graphics/Data/Sine/GradErL2.csv};
\node[anchor=west] (source) at (axis cs:1.1,-4*10e-5){\color{teal}$\epsilon_{\vpinn}'$};
       \node (destination) at (axis cs:1,0){};
       \draw[teal, ->](source)--(destination);
\addplot [mark=none, red, very thick, densely dotted] table [y=b, x=a, col sep = comma,mark=none]{Graphics/Data/Sine/GradErCol.csv};
\node[anchor=west] (source) at (axis cs:2,-4*10e-5){\color{red}$\epsilon_{col}'$};
       \node (destination) at (axis cs:1.8,0){};
       \draw[red, ->](source)--(destination);
\end{axis}
\end{tikzpicture}
\caption{Errors in the gradient}
\end{center}
\end{subfigure}
\end{tabular}
\caption{ Model Problem 1. Obtained solutions and relative $H^1$-error evolution for the three methods}
\label{figMP1Sols}
\end{figure}

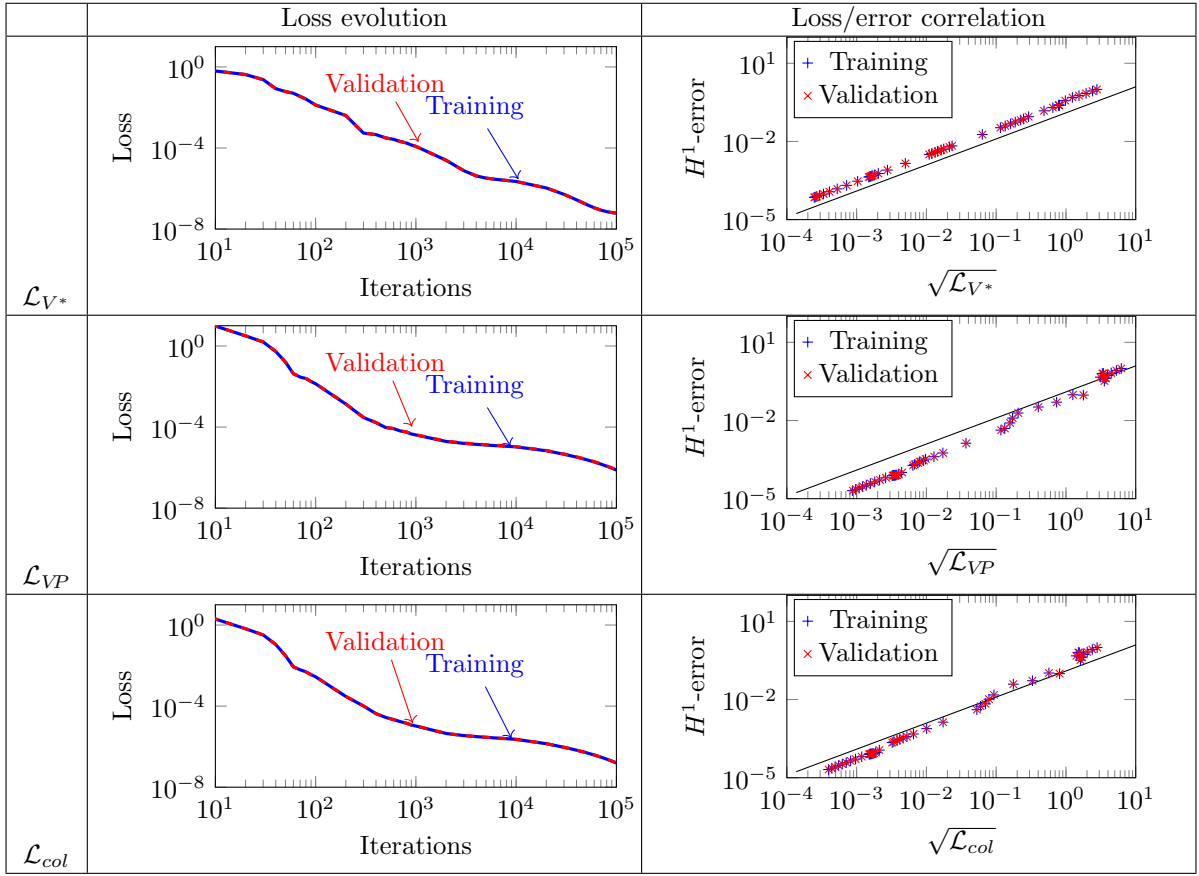
\begin{figure}
\begin{center}
\makebox[\textwidth][c]{\begin{tabular}{|  c | c | c |}
\hline
& Loss evolution & Loss/error correlation\\
\hline
 $\mathcal{L}_{V^*}$&
\begin{subfigure}[b]{0.45\textwidth}
\begin{center}

\begin{tikzpicture}
\begin{axis}[
  xlabel=Iterations,
  ylabel=Loss,
xmin =10e0,
xmax=10e4,
ymin=10e-9,
ymax=10e0,
height=4cm,
width=\textwidth,
xmode=log,
ymode=log]
\addplot [very thick, blue, mark=none]table [y expr=\thisrow{b},x =a, col sep = comma]{Graphics/Data/Sine/LossH1Dual.csv};
\node[anchor=west] (source) at (axis cs:10e2,0.01){\color{blue}Training};
       \node (destination) at (axis cs:1.2*10e3,4*25*10e-9){};
       \draw[blue, ->](source)--(destination);
\addplot[very thick, red, mark=none,dashed]table [y expr=\thisrow{b},x =a, col sep = comma]{Graphics/Data/Sine/ValH1Dual.csv};
\node[anchor=west] (source) at (axis cs:10e1,0.16){\color{red}Validation};
       \node (destination) at (axis cs:1.2*10e2,1.4*5*10e-6){};
       \draw[red, ->](source)--(destination);

\end{axis}
\end{tikzpicture}
\end{center}
\end{subfigure}&
\begin{subfigure}[b]{0.45\textwidth}
\begin{center}
\begin{tikzpicture}
\begin{axis}[
  xlabel=$\sqrt{\mathcal{L}_{V^*}}$,
  ylabel=$H^1$-error,
xmin =10e-5,
xmax=10e0,
ymin=10e-6,
ymax=10e1,
height=4cm,
width=0.9\textwidth,
xmode=log,
ymode=log,
legend style={at={(0.02,0.98)},anchor=north west}]
\addplot [only marks, blue, mark size=2pt, mark=+]table [y expr=sqrt(\thisrow{c}),x expr=sqrt(\thisrow{b}), col sep = comma]{Graphics/Data/Sine/LossH1Dual.csv};
\addlegendentry{Training}
\addplot[only marks, red, mark size=2pt, mark=x]table [y expr=sqrt(\thisrow{c}), ,x expr=sqrt(\thisrow{b}), col sep = comma]{Graphics/Data/Sine/ValH1Dual.csv};
\addlegendentry{Validation}
\node (s1) at (axis cs:10e-5,10e-5/8){};
\node (s2) at (axis cs:10e2,10e2/8){};
\draw (s1)--(s2);
\end{axis}
\end{tikzpicture}
\end{center}
\end{subfigure}\\
\hline
 $\mathcal{L}_{\vpinn}$&
\begin{subfigure}[b]{0.45\textwidth}
\begin{center}
\begin{tikzpicture}
\begin{axis}[
  xlabel=Iterations,
  ylabel=Loss,
xmin =10e0,
xmax=10e4,
ymin=10e-9,
ymax=10e0,
height=4cm,
width=\textwidth,
xmode=log,
ymode=log]
\addplot [very thick, mark=none,blue]table [y expr=\thisrow{b},x =a, col sep = comma]{Graphics/Data/Sine/LossH1L2.csv};
\node[anchor=west] (source) at (axis cs:10e2,0.01){\color{blue}Training};
       \node (destination) at (axis cs:10e3,4*10e-7){};
       \draw[blue, ->](source)--(destination);
\addplot[very thick, mark=none,red,dashed ]table [y expr=\thisrow{b},x =a, col sep = comma]{Graphics/Data/Sine/ValH1L2.csv};
\node[anchor=west] (source) at (axis cs:10e1,0.16){\color{red}Validation};
       \node (destination) at (axis cs:10e2,36*10e-7){};
       \draw[red, ->](source)--(destination);

\end{axis}
\end{tikzpicture}
\end{center}
\end{subfigure}&
\begin{subfigure}[b]{0.45\textwidth}
\begin{center}
\begin{tikzpicture}
\begin{axis}[
  xlabel=$\sqrt{\mathcal{L}_{\vpinn}}$,
  ylabel=$H^1$-error,
xmin =10e-5,
xmax=10e0,
ymin=10e-6,
ymax=10e1,
height=4cm,
width=0.9\textwidth,
xmode=log,
ymode=log,
legend style={at={(0.02,0.98)},anchor=north west}]
\addplot [only marks, blue, mark size=2pt, mark=+]table [y expr=sqrt(\thisrow{c}),x expr=sqrt(\thisrow{b}), col sep = comma]{Graphics/Data/Sine/LossH1L2.csv};
\addlegendentry{Training}
\addplot[only marks, red, mark size=2pt, mark=x]table [y expr=sqrt(\thisrow{c}), ,x expr=sqrt(\thisrow{b}), col sep = comma]{Graphics/Data/Sine/ValH1L2.csv};
\addlegendentry{Validation}
\node (s1) at (axis cs:10e-5,10e-5/8){};
\node (s2) at (axis cs:10e2,10e2/8){};
\draw (s1)--(s2);
\end{axis}
\end{tikzpicture}
\end{center}
\end{subfigure}\\
\hline
 $\mathcal{L}_{col}$&
\begin{subfigure}[b]{0.45\textwidth}
\begin{center}
\begin{tikzpicture}
\begin{axis}[
  xlabel=Iterations,
  ylabel=Loss,
xmin =10e0,
xmax=10e4,
ymin=10e-9,
ymax=10e0,
height=4cm,
width=\textwidth,
xmode=log,
ymode=log]
\addplot [very thick, blue, mark=none]table [y expr=\thisrow{b},x =a, col sep = comma]{Graphics/Data/Sine/LossH1Col.csv};
\node[anchor=west] (source) at (axis cs:10e2,0.01){\color{blue}Training};
       \node (destination) at (axis cs:10e3,9*10e-8){};
       \draw[blue, ->](source)--(destination);
\addplot[red,very thick, mark=none,dashed]table [y expr=\thisrow{b},x =a, col sep = comma]{Graphics/Data/Sine/ValH1Col.csv};
\node[anchor=west] (source) at (axis cs:10e1,0.16){\color{red}Validation};
       \node (destination) at (axis cs:10e2,50*10e-8){};
       \draw[red, ->](source)--(destination);

\end{axis}
\end{tikzpicture}
\end{center}
\end{subfigure}&
\begin{subfigure}[b]{0.45\textwidth}
\begin{center}
\begin{tikzpicture}
\begin{axis}[
  xlabel= $\sqrt{\mathcal{L}_{col}}$,
  ylabel=$H^1$-error,
xmin =10e-5,
xmax=10e0,
ymin=10e-6,
ymax=10e1,
height=4cm,
width=0.9\textwidth,
xmode=log,
ymode=log,
legend style={at={(0.02,0.98)},anchor=north west}
]
\addplot [only marks, blue, mark size=2pt, mark=+]table [y expr=sqrt(\thisrow{c}),x expr=sqrt(\thisrow{b}), col sep = comma]{Graphics/Data/Sine/LossH1Col.csv};
\addlegendentry{Training}
\addplot[only marks, red, mark size=2pt, mark=x]table [y expr=sqrt(\thisrow{c}), ,x expr=sqrt(\thisrow{b}), col sep = comma]{Graphics/Data/Sine/ValH1Col.csv};
\addlegendentry{Validation}
\node (s1) at (axis cs:10e-5,10e-5/8){};
\node (s2) at (axis cs:10e2,10e2/8){};
\draw (s1)--(s2);
\end{axis}
\end{tikzpicture}
\end{center}
\end{subfigure}\\
\hline
\end{tabular}}
\end{center}
\caption{Model Problem 1. Evolution of the loss for the three considered losses on both the training and validation data sets, and the correlation between the loss and the relative $H^1$-error during training, with a straight line corresponding to a linear relationship.}
\label{figMP1Losses}
\end{figure}


\begin{table}[H]
\begin{center}
\begin{tabular}{|c |c |c |c |c|c|c|c|c|c|}
\hline
 & H1 (\%) &L2 (\%) \\
 \hline
$u_{V^*}$ & $7.11\times 10^{-3}$ & $1.26\times 10^{-3}$ \\
\hline
$u_{\vpinn}$ & $2.03\times 10^{-3}$ & $4.56\times 10^{-4}$ \\
\hline
$u_{col}$&  $2.07\times 10^{-3}$ & $5.16\times 10^{-4}$ \\
\hline
\end{tabular}
\end{center}
\caption{Model Problem 1: Errors and losses}
\label{tableStatsMP1}
\end{table}

\Cref{figMP1Sols} shows the obtained solutions and their errors, along with the $H^1$-relative error evolution during training, where we denote errors by $\epsilon$, so that $\epsilon_{X}=u^*-u_X$ for $X=V^*,\vpinn$, and $col$. \Cref{tableStatsMP1} presents the numerical $H^1$- and $L^2$-relative errors of the obtained solutions. We compare the evolution of the loss during training and the relationship between the loss and error in \Cref{figMP1Losses}. 

We observe that all three approximations $u_{V^*},u_{\vpinn}$, and $u_{col}$ converge well to the solution, and quantitatively we see via \Cref{tableStatsMP1} that the relative $H^1$ and $L^2$ errors are small in each case, around $10^{-3}\%$. Generally, all metrics are comparable between the three obtained solutions, as expected by the theory, due to the presence of only low-frequency modes. In particular, we see that the implementations of $\mathcal{L}_{\vpinn}$ and $\mathcal{L}_{\text{col}}$ are almost identical, which is as expected as they can both be interpreted as discretisations of the $L^2$-norm of the strong-form residual. The slight differences in metrics may easily arise from the optimisation procedure, rather than the losses themselves. 

In the column ``Loss/error correlation" of \Cref{figMP1Losses} we show the relationship between the square root of the losses and the relative $H^{-1}$ error during training. As expected, in view of Equation \eqref{eqUpLowerBd}, the square root of the discretised loss $\mathcal{L}_{V^*}$ is an excellent approximation of the $H^{-1}$ norm of the residual. A similar behaviour is observed with the remaining losses: $\mathcal{L}_{\vpinn}$ and $\mathcal{L}_{col}$. However, in these cases, we see slight perturbations in this linear relationship, as expected.

\subsubsection{Model Problem 2 - Large gradients}
\label{ModelProblem2}
Our next model problem is: find $u\in H^1(0,\pi)$ with $u(0)=0$ that satisfies 
\begin{equation}
\int_0^\pi u'(x)v'(x)-2a^2\frac{\tanh\left(a\left(x-\frac{\pi}{2}\right)\right)}{\cosh\left(a\left(x-\frac{\pi}{2}\right)\right)^2}v(x)\,dx +a\text{sech}\left(a\left(\frac{\pi}{2}\right)\right)^2v(\pi)=0
\end{equation}
for all $v\in H^1(0,\pi)$ with $v(0)=0$. The exact solution is given by 
\begin{equation}
u^*(x)=\tanh\left(a\left(x-\frac{\pi}{2}\right)\right)+\tanh\left(\frac{a\pi}{2}\right)
\end{equation}
We consider this example as this admits $C^\infty$ solutions for all $a$, but for $a$ large, a transition region with high gradients develops in the solution. In particular, the forcing term, whilst being a smooth $L^2$ function, has a large discrepancy between its norm in $V^*$ and $L^2$, and thus we expect to see significant differences according to the loss implemented. For our implementation, we take $a=20$. 

As we have a Neumann condition at $x=0$, for the collocation method we need to include a further term to enforce the constraint. Our implementation for the collocation loss $\mathcal{L}_{col}$ uses equal weights for the interior and boundary terms, i.e.

\begin{equation}
\mathcal{L}_{col}(u):=\left|u'(\pi)-a\text{sech}\left(\frac{a\pi}{2}\right)\right|^2+\frac{1}{N}\sum\limits_{i=1}^N |L_u(x_i)|^2,
\end{equation}
where $L_u$ is the strong-form residual. Generally, one could choose to weight the two components of the loss differently, and the choice of weight is an extra parameter which may effect the convergence of the model. An advantage of the weak formulation, however, is that it does not require such a choice. Whilst methods exist to attempt to estimate optimal weights during training within certain settings \cite{wang2022and}, we do not consider them in this work. Despite the need to choose an appropriate weight, we observe in \Cref{figMP2H1Evolve} that there is very little difference between the $H^1$-error evolution using $\mathcal{L}_{\vpinn}$ and $\mathcal{L}_{col}$, suggesting that, in this example, the choice of weight is unimportant.

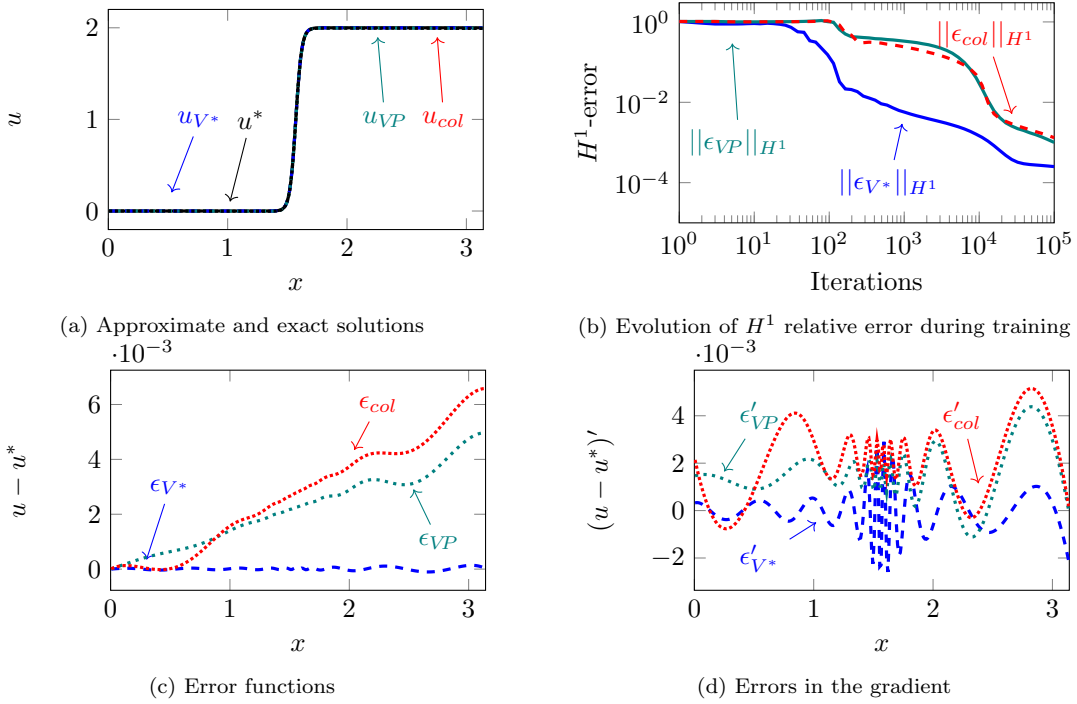
\begin{figure}[H]
\begin{tabular}{c c}
\begin{subfigure}[t]{0.475\textwidth}
\begin{center}
\begin{tikzpicture}
\begin{axis}[
  xlabel=$x$,
  ylabel=$u$,
xmin = 0,
xmax=3.1416,
height=4.5cm,
width=0.9\textwidth]
\addplot [mark=none,very thick] table [y=b, x=a, col sep = comma]{Graphics/Data/Arctan2/SolExact.csv};
\node[anchor=west] (source) at (axis cs:1.0,1){$u^*$};
       \node (destination) at (axis cs:1.,0){};
       \draw[->](source)--(destination);
\addplot [mark=none, blue, very thick,dashed ]table [y=b, x=a, col sep = comma]{Graphics/Data/Arctan2/SolDual.csv};
\node[anchor=west] (source) at (axis cs:0.5,1){\color{blue}$u_{V^*}$};
       \node (destination) at (axis cs:0.5,0.1){};
       \draw[blue, ->](source)--(destination);
\addplot [mark=none, teal!100, very thick,dotted] table [y=b, x=a, col sep = comma]{Graphics/Data/Arctan2/SolL2.csv};
\node[anchor=west] (source) at (axis cs:2.05,1){\color{ teal!100}$u_{\vpinn}$};
       \node (destination) at (axis cs:2.25,2){};
       \draw[teal, ->](source)--(destination);
\addplot [mark=none red, very thick, densely dotted] table [y=b, x=a, col sep = comma,mark=none]{Graphics/Data/Arctan2/SolCol.csv};
\node[anchor=west] (source) at (axis cs:2.55,1){\color{red}$u_{col}$};
       \node (destination) at (axis cs:2.75,2){};
       \draw[red, ->](source)--(destination);
\end{axis}
\end{tikzpicture}
\caption{ Approximate and exact solutions}
\end{center}
\end{subfigure}&
\begin{subfigure}[t]{0.475\textwidth}
\begin{center}
\begin{tikzpicture}
\begin{axis}[
  xlabel=Iterations,
  ylabel=$H^1$-error,
xmin =1,
xmax=10e4,
ymin=10e-6,
ymax=3,
height=4.5cm,
width=0.9\textwidth,
xmode=log,
ymode=log]
\addplot [mark=none, blue, very thick ]table [y expr=sqrt(\thisrow{c}),x=a, col sep = comma]{Graphics/Data/Arctan2/LossH1Dual.csv};
\node[anchor=west] (source) at (axis cs:10e1,010e-5){\color{blue}$||\epsilon_{V^*}||_{H^1}$};
       \node (destination) at (axis cs:10e2,5*10e-4){};
       \draw[blue, ->](source)--(destination);
\addplot [mark=none,teal,very thick ]table [y expr=sqrt(\thisrow{c}),x=a, col sep = comma]{Graphics/Data/Arctan2/LossH1L2.csv};
\node[anchor=west] (source) at (axis cs:1,10e-4){\color{teal}$||\epsilon_{\vpinn}||_{H^1}$};
       \node (destination) at (axis cs:5,5*10e-2){};
       \draw[teal, ->](source)--(destination);
\addplot [mark=none, red, very thick ,dashed]table [y expr=sqrt(\thisrow{c}),x=a, col sep = comma]{Graphics/Data/Arctan2/LossH1Col.csv};
\node[anchor=west] (source) at (axis cs:0.2*10e3,5*10e-2){\color{red}$||\epsilon_{col}||_{H^1}$};
       \node (destination) at (axis cs:3*10e3,0.3*10e-3){};
       \draw[red, ->](source)--(destination);
\end{axis}
\end{tikzpicture}
\caption{Evolution of $H^1$ relative error during training}\label{figMP2H1Evolve}
\end{center}
\end{subfigure}\\
\begin{subfigure}[t]{0.475\textwidth}
\begin{center}
\begin{tikzpicture}
\begin{axis}[
  xlabel=$x$,
  ylabel=$u-u^*$,
xmin = 0,
xmax=3.1416,
height=4.5cm,
width=0.9\textwidth]
\addplot [mark=none, blue, very thick,dashed ]table [y=b, x=a, col sep = comma]{Graphics/Data/Arctan2/ErDual.csv};
\node[anchor=west] (source) at (axis cs:0.25,3*10e-4){\color{blue}$\epsilon_{V^*}$};
       \node (destination) at (axis cs:0.275,0.2*10e-6){};
       \draw[blue, ->](source)--(destination);
\addplot [mark=none, teal, very thick,dotted] table [y=b, x=a, col sep = comma]{Graphics/Data/Arctan2/ErL2.csv};
\node[anchor=west] (source) at (axis cs:2.5,1.*10e-4){\color{teal}$\epsilon_{\vpinn}$};
       \node (destination) at (axis cs:2.5,3.25*10e-4){};
       \draw[teal, ->](source)--(destination);
\addplot [mark=none, red, very thick, densely dotted] table [y=b, x=a, col sep = comma,mark=none]{Graphics/Data/Arctan2/ErCol.csv};
\node[anchor=west] (source) at (axis cs:2,6*10e-4){\color{red}$\epsilon_{col}$};
       \node (destination) at (axis cs:2,4*10e-4){};
       \draw[red, ->](source)--(destination);
\end{axis}
\end{tikzpicture}
\caption{Error functions}
\end{center}
\end{subfigure}&
\begin{subfigure}[t]{0.475\textwidth}
\begin{center}
\begin{tikzpicture}
\begin{axis}[
  xlabel=$x$,
  ylabel=$(u-u^*)'$,
xmin = 0,
xmax=3.1416,
height=4.5cm,
width=0.9\textwidth]
\addplot [mark=none, blue, very thick,dashed ]table [y=b, x=a, col sep = comma]{Graphics/Data/Arctan2/GradErDual.csv};
\node[anchor=west] (source) at (axis cs:0.3,-2*10e-4){\color{blue}$\epsilon_{V^*}'$};
       \node (destination) at (axis cs:1.1,-0.25*10e-4){};
       \draw[blue, ->](source)--(destination);
\addplot [mark=none, teal, very thick,dotted] table [y=b, x=a, col sep = comma]{Graphics/Data/Arctan2/GradErL2.csv};
\node[anchor=west] (source) at (axis cs:0.3,4*10e-4){\color{teal}$\epsilon_{\vpinn}'$};
       \node (destination) at (axis cs:0.2,1.25*10e-4){};
       \draw[teal, ->](source)--(destination);
\addplot [mark=none, red, very thick, densely dotted] table [y=b, x=a, col sep = comma,mark=none]{Graphics/Data/Arctan2/GradErCol.csv};
\node[anchor=west] (source) at (axis cs:2,4*10e-4){\color{red}$\epsilon_{col}'$};
       \node (destination) at (axis cs:2.4,0.75*10e-4){};
       \draw[red, ->](source)--(destination);
\end{axis}
\end{tikzpicture}
\caption{Errors in the gradient}
\end{center}
\end{subfigure}
\end{tabular}
\caption{ Model Problem 2: Obtained solutions and relative $H^1$-error evolution for the three methods}
\label{figMP2Sols}
\end{figure}

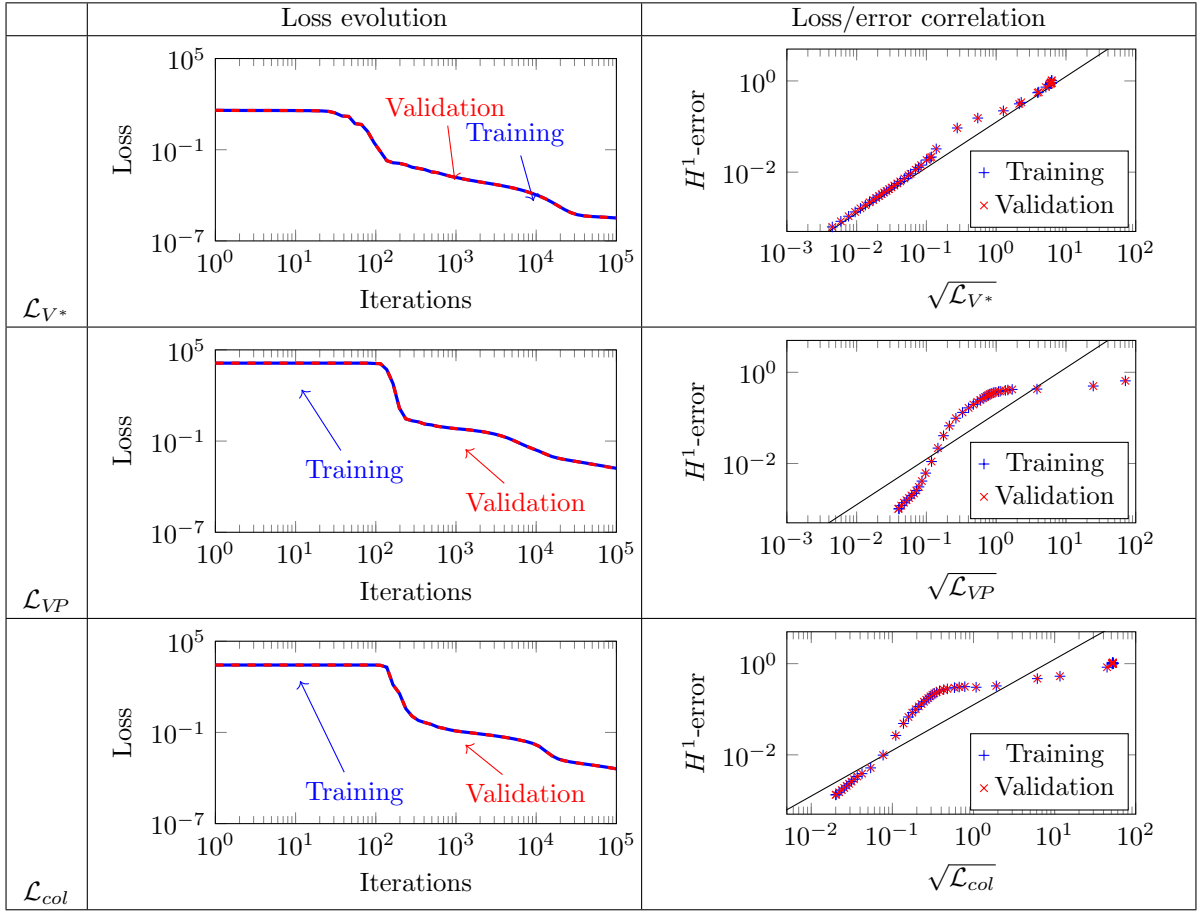
\begin{figure}
\begin{center}
\makebox[\textwidth][c]{\begin{tabular}{|  c | c | c |}
\hline
& Loss evolution & Loss/error correlation\\
\hline
 $\mathcal{L}_{V^*}$&
\begin{subfigure}[b]{0.45\textwidth}
\begin{center}

\begin{tikzpicture}
\begin{axis}[
  xlabel=Iterations,
  ylabel=Loss,
xmin =10e-1,
xmax=10e4,
ymin=10e-8,
ymax=10e4,
height=4cm,
width=\textwidth,
xmode=log,
ymode=log]
\addplot [very thick, blue, mark=none]table [y expr=\thisrow{b},x =a, col sep = comma]{Graphics/Data/Arctan2/LossH1Dual.csv};
\node[anchor=west] (source) at (axis cs:10e2,1){\color{blue}Training};
       \node (destination) at (axis cs:10e3,10e-6){};
       \draw[blue, ->](source)--(destination);
\addplot[very thick, red, mark=none,dashed]table [y expr=\thisrow{b},x =a, col sep = comma]{Graphics/Data/Arctan2/ValH1Dual.csv};
\node[anchor=west] (source) at (axis cs:10e1,64){\color{red}Validation};
       \node (destination) at (axis cs:10e2,25*10e-6){};
       \draw[red, ->](source)--(destination);

\end{axis}
\end{tikzpicture}
\end{center}
\end{subfigure}&
\begin{subfigure}[b]{0.45\textwidth}
\begin{center}
\begin{tikzpicture}
\begin{axis}[
  xlabel=$\sqrt{\mathcal{L}_{V^*}}$,
  ylabel=$H^1$-error,
xmin=0.1*10e-3,
xmax=10e1,
ymin=5*10e-5,
ymax=5*10e-1,
height=4cm,
width=0.9\textwidth,
xmode=log,
ymode=log,
legend style={at={(0.98,0.02)},anchor=south east}]
\addplot [only marks, blue, mark size=2pt, mark=+]table [y expr=sqrt(\thisrow{c}),x expr=sqrt(\thisrow{b}), col sep = comma]{Graphics/Data/Arctan2/LossH1Dual.csv};
\addlegendentry{Training}
\addplot[only marks, red, mark size=2pt, mark=x]table [y expr=sqrt(\thisrow{c}), ,x expr=sqrt(\thisrow{b}), col sep = comma]{Graphics/Data/Arctan2/ValH1Dual.csv};
\addlegendentry{Validation}
\node (s1) at (axis cs:10e-5,10e-5/8){};
\node (s2) at (axis cs:10e2,10e2/8){};
\draw (s1)--(s2);
\end{axis}
\end{tikzpicture}
\end{center}
\end{subfigure}\\
\hline
 $\mathcal{L}_{\vpinn}$&
\begin{subfigure}[b]{0.45\textwidth}
\begin{center}
\begin{tikzpicture}
\begin{axis}[
  xlabel=Iterations,
  ylabel=Loss,
xmin =10e-1,
xmax=10e4,
ymin=10e-8,
ymax=10e4,
height=4cm,
width=\textwidth,
xmode=log,
ymode=log]
\addplot [very thick, mark=none,blue]table [y expr=\thisrow{b},x =a, col sep = comma]{Graphics/Data/Arctan2/LossH1L2.csv};
\node[anchor=west] (source) at (axis cs:10e0,10e-4){\color{blue}Training};
       \node (destination) at (axis cs:10e0,10e2){};
       \draw[blue, ->](source)--(destination);
\addplot[very thick, mark=none,red,dashed ]table [y expr=\thisrow{b},x =a, col sep = comma]{Graphics/Data/Arctan2/ValH1L2.csv};
\node[anchor=west] (source) at (axis cs:10e2,10e-6){\color{red}Validation};
       \node (destination) at (axis cs:10e2,10e-2){};
       \draw[red, ->](source)--(destination);

\end{axis}
\end{tikzpicture}
\end{center}
\end{subfigure}&
\begin{subfigure}[b]{0.45\textwidth}
\begin{center}
\begin{tikzpicture}
\begin{axis}[
  xlabel=$\sqrt{\mathcal{L}_{\vpinn}}$,
  ylabel=$H^1$-error,
height=4cm,
width=0.9\textwidth,
xmode=log,
ymode=log,
xmin=0.1*10e-3,
xmax=10e1,
ymin=5*10e-5,
ymax=5*10e-1,
legend style={at={(0.98,0.02)},anchor=south east}]
\addplot [only marks, blue, mark size=2pt, mark=+]table [y expr=sqrt(\thisrow{c}),x expr=sqrt(\thisrow{b}), col sep = comma]{Graphics/Data/Arctan2/LossH1L2.csv};
\addlegendentry{Training}
\addplot[only marks, red, mark size=2pt, mark=x]table [y expr=sqrt(\thisrow{c}), ,x expr=sqrt(\thisrow{b}), col sep = comma]{Graphics/Data/Arctan2/ValH1L2.csv};
\addlegendentry{Validation}
\node (s1) at (axis cs:10e-5,10e-5/8){};
\node (s2) at (axis cs:10e2,10e2/8){};
\draw (s1)--(s2);
\end{axis}
\end{tikzpicture}
\end{center}
\end{subfigure}\\
\hline
 $\mathcal{L}_{col}$&
\begin{subfigure}[b]{0.45\textwidth}
\begin{center}
\begin{tikzpicture}
\begin{axis}[
  xlabel=Iterations,
  ylabel=Loss,
xmin =10e-1,
xmax=10e4,
ymin=10e-8,
ymax=10e4,
height=4cm,
width=\textwidth,
xmode=log,
ymode=log]
\addplot [very thick, blue, mark=none]table [y expr=\thisrow{b},x =a, col sep = comma]{Graphics/Data/Arctan2/LossH1Col.csv};
\node[anchor=west] (source) at (axis cs:10e0,10e-6){\color{blue}Training};
       \node (destination) at (axis cs:10e0,10e2){};
       \draw[blue, ->](source)--(destination);
\addplot[red,very thick, mark=none,dashed]table [y expr=\thisrow{b},x =a, col sep = comma]{Graphics/Data/Arctan2/ValH1Col.csv};
\node[anchor=west] (source) at (axis cs:10e2,10e-6){\color{red}Validation};
       \node (destination) at (axis cs:10e2,0.49*10e-2){};
       \draw[red, ->](source)--(destination);

\end{axis}
\end{tikzpicture}
\end{center}
\end{subfigure}&
\begin{subfigure}[b]{0.45\textwidth}
\begin{center}
\begin{tikzpicture}
\begin{axis}[
  xlabel= $\sqrt{\mathcal{L}_{col}}$,
  ylabel=$H^1$-error,
height=4cm,
width=0.9\textwidth,
xmode=log,
ymode=log,
xmin=0.5*10e-3,
xmax=10e1,
ymin=5*10e-5,
ymax=5*10e-1,
legend style={at={(0.98,0.02)},anchor=south east}
]
\addplot [only marks, blue, mark size=2pt, mark=+]table [y expr=sqrt(\thisrow{c}),x expr=sqrt(\thisrow{b}), col sep = comma]{Graphics/Data/Arctan2/LossH1Col.csv};
\addlegendentry{Training}
\addplot[only marks, red, mark size=2pt, mark=x]table [y expr=sqrt(\thisrow{c}), ,x expr=sqrt(\thisrow{b}), col sep = comma]{Graphics/Data/Arctan2/ValH1Col.csv};
\addlegendentry{Validation}
\node (s1) at (axis cs:10e-5,10e-5/8){};
\node (s2) at (axis cs:10e2,10e2/8){};
\draw (s1)--(s2);
\end{axis}
\end{tikzpicture}
\end{center}
\end{subfigure}\\
\hline
\end{tabular}}
\end{center}
\caption{Model Problem 2. Evolution of the loss for the three considered losses on both the training and validation data sets, and the correlation between the loss and the relative $H^1$-error during training, with a straight line corresponding to a linear relationship.}
\label{figMP2Losses}
\end{figure}

%

\begin{table}[H]
\begin{center}
\begin{tabular}{|c |c |c |c |c|c|c|c|c|c|}
\hline
&H1 (\%) &L2 (\%)  \\
 \hline
$u_{V^*}$ &  $0.025$ & $0.004$ \\
\hline
$u_{\vpinn}$ &  $0.100$ & $0.184$ \\
\hline
$u_{col}$ &  $0.132$ & $0.242$  \\
\hline
\end{tabular}
\end{center}
\caption{Model problem 2: Errors and losses}
\label{tabMP2Losses}
\end{table}

\Cref{figMP2Sols} shows that all three methods converge to the exact solution, and this is seen quantitatively in \Cref{tabMP2Losses}. We see that $u_{V^*}$ shows the best performance, both in terms of speed of convergence and the error of the solution at the end of training. Similar to Model Problem 1, we see that $\mathcal{L}_{\vpinn}$ and $\mathcal{L}_{col}$ show similar behaviours in all regards. As before, the correlation between the $H^1$-error and the square root of $\mathcal{L}_{V^*}$ is extremely strong, showing a directly proportional relationship between them. This behaviour however is no longer seen when $\mathcal{L}_{\vpinn}$ and $\mathcal{L}_{col}$ are implemented. In particular, we observe that during the initial training, $\mathcal{L}_{\vpinn}$ and $\mathcal{L}_{col}$ decrease by several orders of magnitude before the error itself starts to decrease significantly. Thus, we demonstrate that even in the case of a highly regular problem with $C^\infty$ solution, both $\mathcal{L}_{col}$ and $\mathcal{L}_{\vpinn}$ can fail to be good estimators of the $H^1$-error, whilst the DFR method, due to its directly proportional relationship, permits good $H^1$-error estimation.

%

\subsubsection{Model Problem 3 - Discontinuous parameters}
\label{ModelProblem3}

We next consider an ODE with discontinuous parameters, whose solution is in $H^1$ but is not $C^1$ nor $H^2$. We take $\Gamma_D=\{0,\pi\}$ and aim to solve 

\begin{equation}
\int_0^\pi \sigma(x)u'(x)v'(x)-4\sin(2x)v(x)\,dx=0
\end{equation}
for all $v\in H^1_0(0,\pi)$, where 

\begin{equation}\label{eqStrongSolSigmaDiscSoln}
\sigma(x)=\left\{\begin{array}{c c}
1 & x<\frac{\pi}{2},\\
2  & x>\frac{\pi}{2}.
\end{array}\right.
\end{equation}
The exact solution to this problem is given by
\begin{equation}
u^*(x)=\left\{\begin{array}{c c}
\sin 2x & x<\frac{\pi}{2},\\
\frac{1}{2}\sin 2x  & x>\frac{\pi}{2}.
\end{array}\right.
\end{equation}
In particular, $u^*$ admits a jump discontinuity in its gradient at $x=\frac{\pi}{2}$.

Since $\sigma$ is discontinuous, the strong and weak forms are not equivalent. In particular, as $\sigma$ is piecewise constant, outside of its single point of discontinuity there is no difference in the (strong) PDEs between 
\begin{equation}
(\sigma(x)u'(x))'+4\sin 2x=0
\end{equation}
and 
\begin{equation}\label{eqStrongSolSigmaDisc}
u''(x)+\frac{4}{\sigma(x)}\sin 2x=0.
\end{equation}
There is a unique $C^1$ function which solves \eqref{eqStrongSolSigmaDisc} on $(0,\pi)\setminus \left\{\frac{\pi}{2}\right\}$, given by 
\begin{equation}\label{eqStrongSolSigmaDiscSoln}
\tilde{u}(x)=\left\{\begin{array}{c c}
\sin 2x+\frac{1}{2}x& x<\frac{\pi}{2},\\
\frac{1}{2}\sin 2x -\frac{1}{2}(x-\pi)  & x>\frac{\pi}{2}.
\end{array}\right.
\end{equation}
We expect the collocation loss $\mathcal{L}_{col}$ to fail, as it is ill-equipped to handle PDEs that lack an equivalent strong formulation. Whilst the discretised VPINN loss $\mathcal{L}_{\vpinn}$ is well defined at any canididate solution, due to the discontinuity in $\sigma$, for a general, smooth, trial function $u$, the series \eqref{eqDefL2Loss} should diverge as the number of basis functions tends to infinity, as the residual cannot generally be expressed as an $L^2$ function, so we expect the results to not be trustworthy.

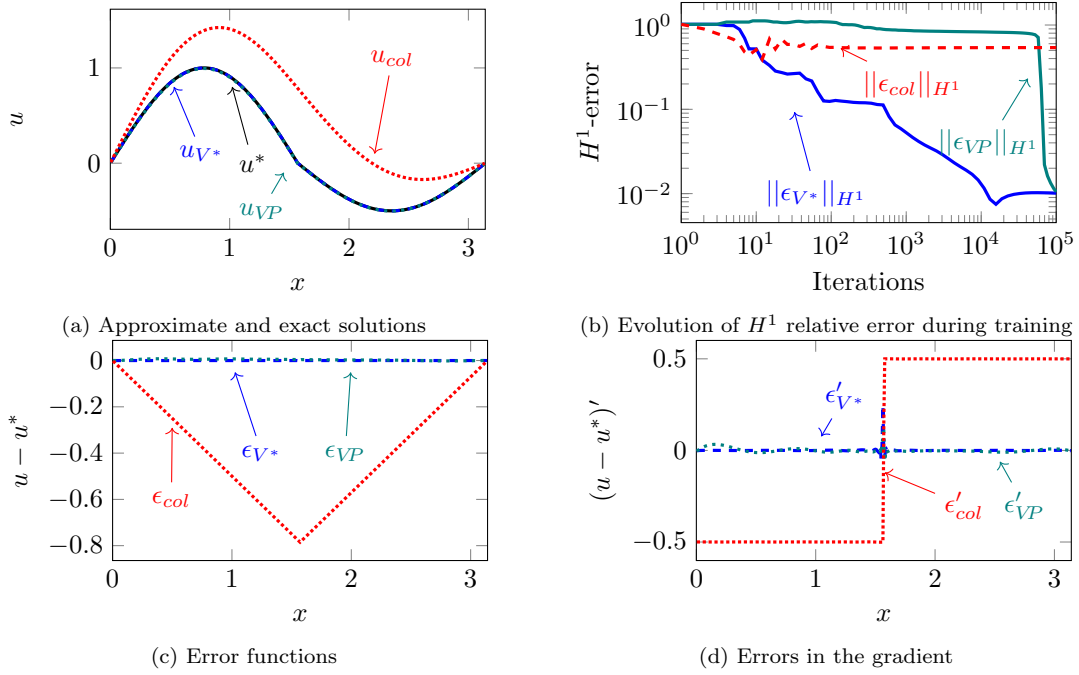
\begin{figure}[H]
\begin{tabular}{c c}
\begin{subfigure}[t]{0.475\textwidth}
\begin{center}
\begin{tikzpicture}
\begin{axis}[
  xlabel=$x$,
  ylabel=$u$,
xmin = 0,
xmax=3.1416,
height=4.5cm,
width=0.9\textwidth]
\addplot [mark=none,very thick] table [y=b, x=a, col sep = comma]{Graphics/Data/DiscSigma2/SolExact.csv};
\node[anchor=west] (source) at (axis cs:1.0,0){$u^*$};
       \node (destination) at (axis cs:1.,0.9){};
       \draw[->](source)--(destination);
\addplot [mark=none, blue, very thick,dashed ]table [y=b, x=a, col sep = comma]{Graphics/Data/DiscSigma2/SolDual.csv};
\node[anchor=west] (source) at (axis cs:0.5,.1){\color{blue}$u_{V^*}$};
       \node (destination) at (axis cs:0.5,0.9){};
       \draw[blue, ->](source)--(destination);
\addplot [mark=none, teal!100, very thick,dotted] table [y=b, x=a, col sep = comma]{Graphics/Data/DiscSigma2/SolL2.csv};
\node[anchor=west] (source) at (axis cs:1,-0.5){\color{ teal!100}$u_{\vpinn}$};
       \node (destination) at (axis cs:1.5,0){};
       \draw[teal, ->](source)--(destination);
\addplot [mark=none, red, very thick, densely dotted] table [y=b, x=a, col sep = comma,mark=none]{Graphics/Data/DiscSigma2/SolCol.csv};
\node[anchor=west] (source) at (axis cs:2.1,1.1){\color{red}$u_{col}$};
       \node (destination) at (axis cs:2.2,0){};
       \draw[red, ->](source)--(destination);
\end{axis}
\end{tikzpicture}
\caption{ Approximate and exact solutions}
\end{center}
\end{subfigure}&
\begin{subfigure}[t]{0.475\textwidth}
\begin{center}
\begin{tikzpicture}
\begin{axis}[
  xlabel=Iterations,
  ylabel=$H^1$-error,
xmin =1,
xmax=10e4,
height=4.5cm,
width=0.9\textwidth,
xmode=log,
ymode=log]
\addplot [mark=none, blue, very thick ]table [y expr=sqrt(\thisrow{c}),x=a, col sep = comma]{Graphics/Data/DiscSigma2/LossH1Dual.csv};
\node[anchor=west] (source) at (axis cs:10e0,10e-3){\color{blue}$||\epsilon_{V^*}||_{H^1}$};
       \node (destination) at (axis cs:3*10e0,1.2*10e-2){};
       \draw[blue, ->](source)--(destination);
\addplot [mark=none,teal,very thick ]table [y expr=sqrt(\thisrow{c}),x=a, col sep = comma]{Graphics/Data/DiscSigma2/LossH1L2.csv};
\node[anchor=west] (source) at (axis cs:0.2*10e3,0.4*10e-2){\color{teal}$||\epsilon_{\vpinn}||_{H^1}$};
       \node (destination) at (axis cs:5*10e3,0.3*10e-1){};
       \draw[teal, ->](source)--(destination);
\addplot [mark=none, red, very thick ,dashed]table [y expr=sqrt(\thisrow{c}),x=a, col sep = comma]{Graphics/Data/DiscSigma2/LossH1Col.csv};
\node[anchor=west] (source) at (axis cs:0.2*10e2,2*10e-2){\color{red}$||\epsilon_{col}||_{H^1}$};
       \node (destination) at (axis cs:10e1,6*10e-2){};
       \draw[red, ->](source)--(destination);
\end{axis}
\end{tikzpicture}
\caption{Evolution of $H^1$ relative error during training}
\end{center}
\end{subfigure}\\
\begin{subfigure}[t]{0.475\textwidth}
\begin{center}
\begin{tikzpicture}
\begin{axis}[
  xlabel=$x$,
  ylabel=$u-u^*$,
xmin = 0,
xmax=3.1416,
height=4.5cm,
width=0.9\textwidth]
\addplot [mark=none, blue, very thick,dashed ]table [y=b, x=a, col sep = comma]{Graphics/Data/DiscSigma2/ErDual.csv};
\node[anchor=west] (source) at (axis cs:1,-0.4){\color{blue}$\epsilon_{V^*}$};
       \node (destination) at (axis cs:1,-0.01){};
       \draw[blue, ->](source)--(destination);
\addplot [mark=none, teal, very thick,dotted] table [y=b, x=a, col sep = comma]{Graphics/Data/DiscSigma2/ErL2.csv};
\node[anchor=west] (source) at (axis cs:1.7,-0.4){\color{teal}$\epsilon_{\vpinn}$};
       \node (destination) at (axis cs:2,-0.01){};
       \draw[teal, ->](source)--(destination);
\addplot [mark=none, red, very thick, densely dotted] table [y=b, x=a, col sep = comma,mark=none]{Graphics/Data/DiscSigma2/ErCol.csv};
\node[anchor=west] (source) at (axis cs:0.25,-0.6){\color{red}$\epsilon_{col}$};
       \node (destination) at (axis cs:0.5,-0.25){};
       \draw[red, ->](source)--(destination);
\end{axis}
\end{tikzpicture}
\caption{Error functions}
\end{center}
\end{subfigure}&
\begin{subfigure}[t]{0.475\textwidth}
\begin{center}
\begin{tikzpicture}
\begin{axis}[
  xlabel=$x$,
  ylabel=$(u-u^*)'$,
xmin = 0,
xmax=3.1416,
height=4.5cm,
width=0.9\textwidth]
\addplot [mark=none, blue, very thick,dashed ]table [y=b, x=a, col sep = comma]{Graphics/Data/DiscSigma2/GradErDual.csv};
\node[anchor=west] (source) at (axis cs:1,0.3){\color{blue}$\epsilon_{V^*}'$};
       \node (destination) at (axis cs:1,0.025){};
       \draw[blue, ->](source)--(destination);
\addplot [mark=none, teal, very thick,dotted] table [y=b, x=a, col sep = comma]{Graphics/Data/DiscSigma2/GradErL2.csv};
\node[anchor=west] (source) at (axis cs:2.5,-0.3){\color{teal}$\epsilon_{\vpinn}'$};
       \node (destination) at (axis cs:2.5,-0.01){};
       \draw[teal, ->](source)--(destination);
\addplot [mark=none, red, very thick, densely dotted] table [y=b, x=a, col sep = comma,mark=none]{Graphics/Data/DiscSigma2/GradErCol.csv};
\node[anchor=west] (source) at (axis cs:2,-0.3){\color{red}$\epsilon_{col}'$};
       \node (destination) at (axis cs:1.5,-0.1){};
       \draw[red, ->](source)--(destination);
\end{axis}
\end{tikzpicture}
\caption{Errors in the gradient}
\end{center}
\end{subfigure}
\end{tabular}
\caption{ Model Problem 3: Obtained solutions and relative $H^1$-error evolution for the three methods}
\label{figMP3Sols}
\end{figure}
	

\begin{figure}
\begin{center}
\makebox[\textwidth][c]{\begin{tabular}{|  c | c | c |}
\hline
& Loss evolution & Loss/error correlation\\
\hline
 $\mathcal{L}_{V^*}$&
\begin{subfigure}[b]{0.45\textwidth}
\begin{center}

\begin{tikzpicture}
\begin{axis}[
  xlabel=Iterations,
  ylabel=Loss,
xmin =10e-1,
xmax=10e4,
ymin=10e-8,
ymax=25*10e0,
height=4cm,
width=\textwidth,
xmode=log,
ymode=log]
\addplot [very thick, blue, mark=none]table [y expr=\thisrow{b},x =a, col sep = comma]{Graphics/Data/DiscSigma2/LossH1Dual.csv};
\node[anchor=west] (source) at (axis cs:10e2,25){\color{blue}Training};
       \node (destination) at (axis cs:10e3,10e-4){};
       \draw[blue, ->](source)--(destination);
\addplot[very thick, red, mark=none,dashed]table [y expr=\thisrow{b},x =a, col sep = comma]{Graphics/Data/DiscSigma2/ValH1Dual.csv};
\node[anchor=west] (source) at (axis cs:10e0,0.0001){\color{red}Validation};
       \node (destination) at (axis cs:10e2,0.01){};
       \draw[red, ->](source)--(destination);

\end{axis}
\end{tikzpicture}
\end{center}
\end{subfigure}&
\begin{subfigure}[b]{0.45\textwidth}
\begin{center}
\begin{tikzpicture}
\begin{axis}[
  xlabel=$\sqrt{\mathcal{L}_{V^*}}$,
  ylabel=$H^1$-error,
height=4cm,
width=0.9\textwidth,
xmode=log,
ymode=log,
xmin=10e-4,
xmax=10e0,
ymin=10e-4,
ymax=5*10e-1,
legend style={at={(0.02,0.98)},anchor=north west}]
\addplot [only marks, blue, mark size=2pt, mark=+]table [y expr=sqrt(\thisrow{c}),x expr=sqrt(\thisrow{b}), col sep = comma]{Graphics/Data/DiscSigma2/LossH1Dual.csv};
\addlegendentry{Training}
\addplot[only marks, red, mark size=2pt, mark=x]table [y expr=sqrt(\thisrow{c}), ,x expr=sqrt(\thisrow{b}), col sep = comma]{Graphics/Data/DiscSigma2/ValH1Dual.csv};
\addlegendentry{Validation}
\node (s1) at (axis cs:10e-5,10e-5/8){};
\node (s2) at (axis cs:10e2,10e2/8){};
\draw (s1)--(s2);
\end{axis}
\end{tikzpicture}
\end{center}
\end{subfigure}\\
\hline
 $\mathcal{L}_{\vpinn}$&
\begin{subfigure}[b]{0.45\textwidth}
\begin{center}
\begin{tikzpicture}
\begin{axis}[
  xlabel=Iterations,
  ylabel=Loss,
xmin =10e-1,
xmax=10e4,
ymin=10e-8,
ymax=25*10e0,
height=4cm,
width=\textwidth,
xmode=log,
ymode=log]
\addplot [very thick, mark=none,blue]table [y expr=\thisrow{b},x =a, col sep = comma]{Graphics/Data/DiscSigma2/LossH1L2.csv};
\node[anchor=west] (source) at (axis cs:10e0,25*10e-6){\color{blue}Training};
       \node (destination) at (axis cs:10e0,10e1){};
       \draw[blue, ->](source)--(destination);
\addplot[very thick, mark=none,red,dashed ]table [y expr=\thisrow{b},x =a, col sep = comma]{Graphics/Data/DiscSigma2/ValH1L2.csv};
\node[anchor=west] (source) at (axis cs:10e2,49*10e-6){\color{red}Validation};
       \node (destination) at (axis cs:10e2,10){};
       \draw[red, ->](source)--(destination);

\end{axis}
\end{tikzpicture}
\end{center}
\end{subfigure}&
\begin{subfigure}[b]{0.45\textwidth}
\begin{center}
\begin{tikzpicture}
\begin{axis}[
  xlabel=$\sqrt{\mathcal{L}_{\vpinn}}$,
  ylabel=$H^1$-error,
height=4cm,
width=0.9\textwidth,
xmode=log,
ymode=log,
xmin=10e-4,
xmax=10e0,
ymin=10e-4,
ymax=5*10e-1,
legend style={at={(0.02,0.98)},anchor=north west}]
\addplot [only marks, blue, mark size=2pt, mark=+]table [y expr=sqrt(\thisrow{c}),x expr=sqrt(\thisrow{b}), col sep = comma]{Graphics/Data/DiscSigma2/LossH1L2.csv};
\addlegendentry{Training}
\addplot[only marks, red, mark size=2pt, mark=x]table [y expr=sqrt(\thisrow{c}), ,x expr=sqrt(\thisrow{b}), col sep = comma]{Graphics/Data/DiscSigma2/ValH1L2.csv};
\addlegendentry{Validation}
\node (s1) at (axis cs:10e-5,10e-5/8){};
\node (s2) at (axis cs:10e2,10e2/8){};
\draw (s1)--(s2);
\end{axis}
\end{tikzpicture}
\end{center}
\end{subfigure}\\
\hline
 $\mathcal{L}_{col}$&
\begin{subfigure}[b]{0.45\textwidth}
\begin{center}
\begin{tikzpicture}
\begin{axis}[
  xlabel=Iterations,
  ylabel=Loss,
xmin =10e-1,
xmax=10e4,
ymin=10e-8,
ymax=25*10e0,
height=4cm,
width=\textwidth,
xmode=log,
ymode=log]
\addplot [very thick, blue, mark=none]table [y expr=\thisrow{b},x =a, col sep = comma]{Graphics/Data/DiscSigma2/LossH1Col.csv};
\node[anchor=west] (source) at (axis cs:10e0,10e-6){\color{blue}Training};
       \node (destination) at (axis cs:10e0,5*10e-2){};
       \draw[blue, ->](source)--(destination);
\addplot[red,very thick, mark=none,dashed]table [y expr=\thisrow{b},x =a, col sep = comma]{Graphics/Data/DiscSigma2/ValH1Col.csv};
\node[anchor=west] (source) at (axis cs:10e2,10e0){\color{red}Validation};
       \node (destination) at (axis cs:10e2,3*10e-4){};
       \draw[red, ->](source)--(destination);

\end{axis}
\end{tikzpicture}
\end{center}
\end{subfigure}&
\begin{subfigure}[b]{0.45\textwidth}
\begin{center}
\begin{tikzpicture}
\begin{axis}[
  xlabel= $\sqrt{\mathcal{L}_{col}}$,
  ylabel=$H^1$-error,
height=4cm,
width=0.9\textwidth,
xmode=log,
ymode=log,
xmin=10e-4,
xmax=10e0,
ymin=10e-4,
ymax=5*10e-1,
legend style={at={(0.98,0.02)},anchor=south east}
]
\addplot [only marks, blue, mark size=2pt, mark=+]table [y expr=sqrt(\thisrow{c}),x expr=sqrt(\thisrow{b}), col sep = comma]{Graphics/Data/DiscSigma2/LossH1Col.csv};
\addlegendentry{Training}
\addplot[only marks, red, mark size=2pt, mark=x]table [y expr=sqrt(\thisrow{c}), ,x expr=sqrt(\thisrow{b}), col sep = comma]{Graphics/Data/DiscSigma2/ValH1Col.csv};
\addlegendentry{Validation}
\node (s1) at (axis cs:10e-5,10e-5/8){};
\node (s2) at (axis cs:10e2,10e2/8){};
\draw (s1)--(s2);
\end{axis}
\end{tikzpicture}
\end{center}
\end{subfigure}\\
\hline
\end{tabular}}
\end{center}
\caption{Model Problem 3. Evolution of the loss for the three considered losses on both the training and validation data sets, and the correlation between the loss and the relative $H^1$-error during training, with a straight line corresponding to a linear relationship.}
\label{figMP3Losses}
\end{figure}
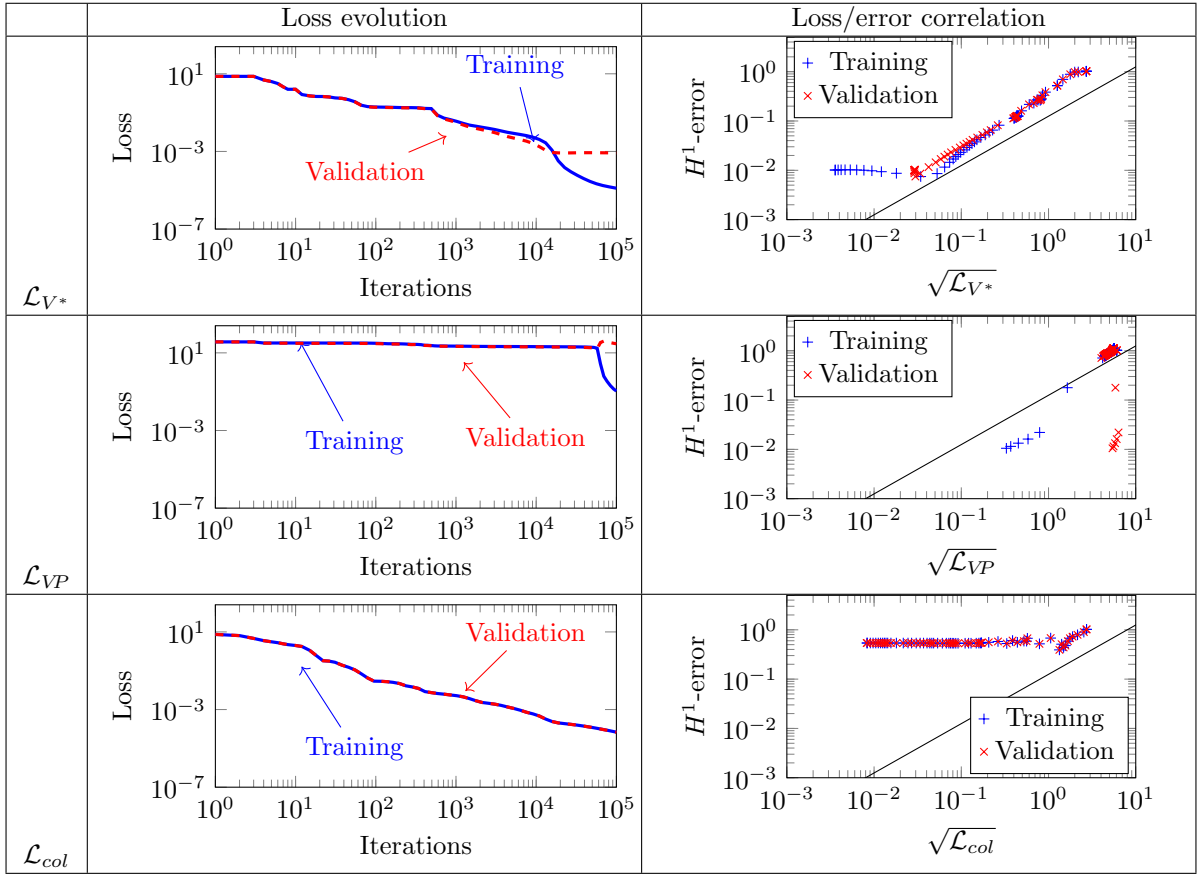


\begin{table}[h!]
\begin{center}
\begin{tabular}{|c |c |c |c |c|c|c|c|c|c|}
\hline
 &H1 (\%) &L2 (\%)  \\
 \hline
$u_{V^*}$ & $1.01$ & $0.0147$  \\
\hline
$u_{\vpinn}$ & $0.988$ & $0.751$  \\
\hline
$u_{col}$ & $54.0$ & $81.1$ \\
\hline
\end{tabular}
\end{center}
\caption{Model Problem 3: Errors and losses}
\label{tabDiscSigma}
\end{table}

Both qualitatively in \Cref{figMP3Sols} and quantitatively in \Cref{tabDiscSigma} we see that $u_{V^*}$ produces a good approximation of the exact solution. \Cref{figMP3Losses} shows overfitting during the training of $u_{V^*}$ at around $2\times 10^4$ iterations, as shown by the divergence of the loss on the training and validation sets. At this point, the $H^1$ relative error stagnates and ceases to decrease significantly. Before overfitting occurs, we observe a perfect linear relationship between the square root of the loss on the training data, however the validation loss remains directly proportional until an uptick corresponding to the region where the validation loss plateaus. This is also the point at which the $H^1$-error reaches its minimum, and later begins to increase. 

Unsurprisingly, we see that $u_{col}$ approximates \eqref{eqStrongSolSigmaDiscSoln}, rather than $u^*$, as the loss implemented corresponds precisely to \eqref{eqStrongSolSigmaDisc}, and thus produces a very poor solution. In contrast to the previous examples, as the residual is generally not expressable as a function in $L^2$, we observe a very large discrepancy between the behaviour of $\mathcal{L}_{V^*}$ and $\mathcal{L}_{\vpinn}$.

Furthermore, during the training of $u_{\vpinn}$, in \Cref{figMP3Losses} we see what would appear to be an extreme case of overfitting due to a large discrepency between the loss evaluated on the training and validation set. However this does not translate into errors, and by comparing the $H^1$-error evolution in \Cref{figMP3Sols} with the loss evolution in \Cref{figMP3Losses}, we see that precisely at the point during training where this ``overfitting" takes place, around $5\times 10^5$ iterations, the $H^1$-error begins to drop significantly and we obtain a good approximation to $u^*$. This is not, however, paradoxical, as we know that the DCT/DST are exact when only low-frequency Fourier modes are present. Thus a large discrepency between the loss evaluated on a training and validation set, which employ both distinct integration points and distinct cutoff frequencies, implies the presence of high-order Fourier modes in the residual. The smoothing effect of the PDE solution operator, however, mitigates the influence of the high-order modes in the residual on the $H^1$-error. Whilst we obtain a good solution, without having the exact solution at hand, it would not be clear if this overfitting is problematic or not without resorting to some other method to attempt to quantify the error. Furthermore, if training had been stopped when this overfitting began to develop, as is traditionally done, we would obtain a solution with relative $H^1$-error close to $100\%$. Due to this, we conclude that, despite $\mathcal{L}_{\vpinn}$ being a well-defined loss for PDEs in weak form, it is inappropriate to use when the weak and strong forms are non-equivalent as one cannot relate the loss on training/validation sets to errors in a clear way, just as $\mathcal{L}_{col}$ would be inappropriate in the same situation. This shows the advantage of employing the DFR method in problems where solutions admit only $H^1$ regularity, making the $H^{-1}$-norm of the residual the appropriate loss function to be minimised. 

%
%

\subsection{Further Results}\label{secExperiments}
We have seen in Model Problem 3 that a validation set is necessary, as we can identify overfitting via a divergence in the loss evaluated on the training and validation sets. For the following examples, inspired by this, we implement an {\it EarlyStopping} callback to stop training when the loss evaluated on the validation set does not show improvement during 200 iterations and restores the best NN parameters according to the best obtained value of the loss evaluated on the validation set. In the following, we perform $10^5$ iterations, or until the {\it EarlyStopping} halts training. With only these exceptions, we consider the same architectures and optimisation proceedures as before.

\subsubsection{Model Problem 4: Point source}\label{ModelProblemDelta}

We take $V=H^1_0(0,\pi)$ and aim to find $u\in V$ such that 
\begin{equation}
\int_0^\pi u'(x)v'(x)\,dx -v\left(\frac{\pi}{2}\right)=0
\end{equation}
for all $v\in V$. This has a unique solution given by 
\begin{equation}
u^*(x)=\frac{\pi}{2}-\left|x-\frac{\pi}{2}\right|.
\end{equation}
The forcing term, given by a Dirac delta function, is in $V^*$ but not expressable as an $L^2$ function. In particular, it would be impossible to solve this equation using classical PINN methods.

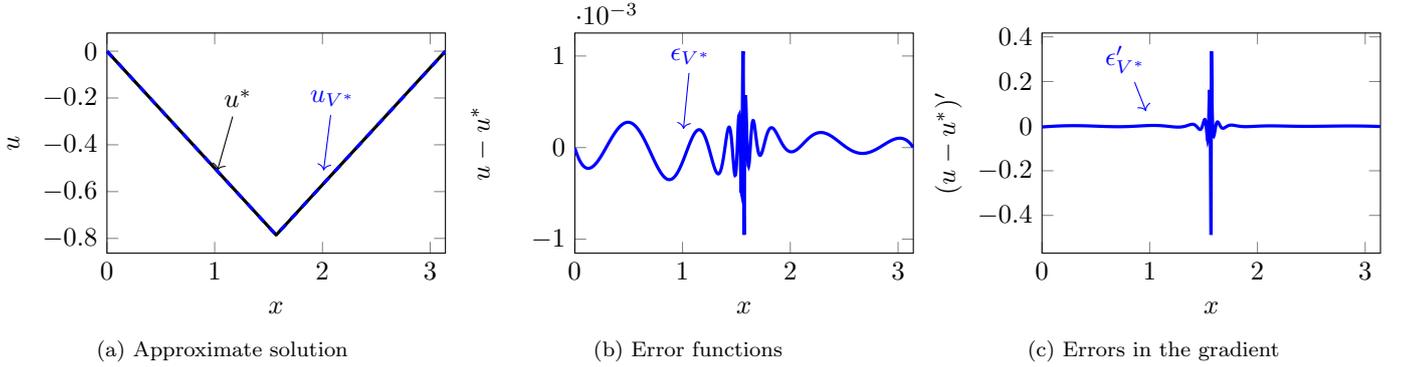
\begin{figure}[H]
\begin{center}\centerline{\begin{minipage}{1.2\textwidth}
\begin{subfigure}[b]{0.33\textwidth}
\begin{center}
\begin{tikzpicture}
\begin{axis}[
  xlabel=$x$,
  ylabel=$u$,
xmin = 0,
xmax=3.1416,
height=4.5cm,
width=\textwidth]
\addplot [mark=none,very thick] table [y=b, x=a, col sep = comma]{Graphics/Data/Delta/SolExact.csv};
\node[anchor=west] (source) at (axis cs:1.0,-0.2){$u^*$};
       \node (destination) at (axis cs:1.,-0.55){};
       \draw[->](source)--(destination);
\addplot [mark=none, blue, very thick,dashed ]table [y=b, x=a, col sep = comma]{Graphics/Data/Delta/SolDual.csv};
\node[anchor=west] (source) at (axis cs:1.8,-0.2){\color{blue}$u_{V^*}$};
       \node (destination) at (axis cs:2,-0.55){};
       \draw[blue, ->](source)--(destination);

\end{axis}
\end{tikzpicture}
\caption{ Approximate solution}
\end{center}
\end{subfigure}
\begin{subfigure}[b]{0.33\textwidth}
\begin{center}
\begin{tikzpicture}
\begin{axis}[
  xlabel=$x$,
  ylabel=$u-u^*$,
xmin = 0.001,
xmax=3.1414,
height=4.5cm,
width=\textwidth]
\addplot [mark=none, blue, very thick ]table [y=b, x=a, col sep = comma]{Graphics/Data/Delta/ErDual.csv};
\node[anchor=west] (source) at (axis cs:0.8,10e-4){\color{blue}$\epsilon_{V^*}$};
       \node (destination) at (axis cs:1,10e-5){};
       \draw[blue, ->](source)--(destination);

\end{axis}
\end{tikzpicture}
\caption{Error functions}
\end{center}
\end{subfigure}
\begin{subfigure}[b]{0.33\textwidth}
\begin{center}
\begin{tikzpicture}
\begin{axis}[
  xlabel=$x$,
  ylabel=$(u-u^*)'$,
xmin = 0,
xmax=3.1416,
height=4.5cm,
width=\textwidth]
\addplot [mark=none, blue, very thick ]table [y=b, x=a, col sep = comma]{Graphics/Data/Delta/GradErDual.csv};
\node[anchor=west] (source) at (axis cs:0.5,0.3){\color{blue}$\epsilon_{V^*}'$};
       \node (destination) at (axis cs:1,0.025){};
       \draw[blue, ->](source)--(destination);
\end{axis}
\end{tikzpicture}
\caption{Errors in the gradient}
\end{center}
\end{subfigure}
\end{minipage}}
\end{center}
\caption{ Model Problem 4. Obtained solution and error}
\label{figMP5sols}
\end{figure}

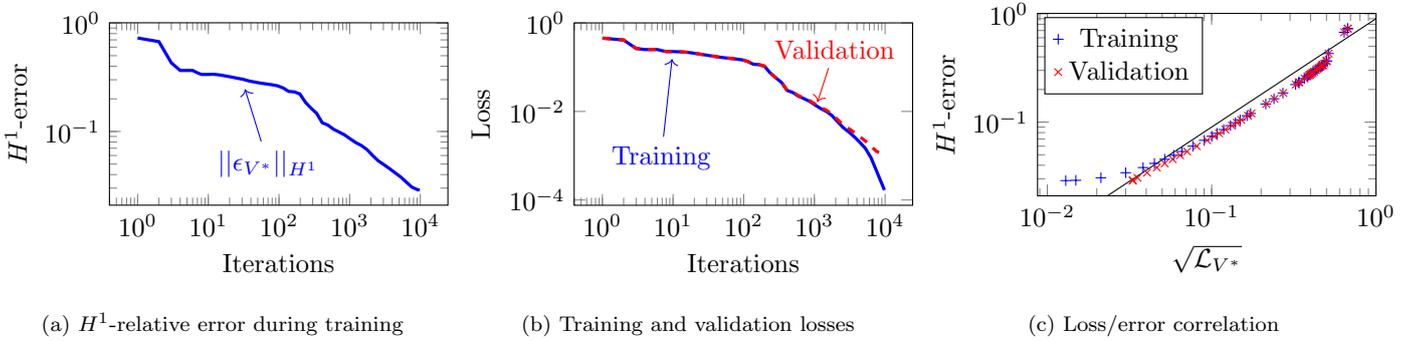
\begin{figure}[H]\centerline{\begin{minipage}{1.2\textwidth}
\begin{subfigure}[b]{0.33\textwidth}
\begin{center}
\begin{tikzpicture}
\begin{axis}[
  xlabel=Iterations,
  ylabel=$H^1$-error,
height=4cm,
width=\textwidth,
xmode=log,
ymode=log]
\addplot [mark=none, blue, very thick ]table [y expr=sqrt(\thisrow{c}),x=a, col sep = comma]{Graphics/Data/Delta/LossH1Dual.csv};
\node[anchor=west] (source) at (axis cs:10e0,0.5*10e-2){\color{blue}$||\epsilon_{V^*}||_{H^1}$};
       \node (destination) at (axis cs:3*10e0,3*10e-2){};
       \draw[blue, ->](source)--(destination);

\end{axis}
\end{tikzpicture}
\end{center}
\caption{$H^1$-relative error during training}\label{subfigMP5h1evolve}
\end{subfigure}
\begin{subfigure}[b]{0.33\textwidth}
\begin{center}
\begin{tikzpicture}
\begin{axis}[
  xlabel=Iterations,
  ylabel=Loss,
height=4cm,
width=\textwidth,
xmode=log,
ymode=log]
\addplot [very thick, mark=none,blue]table [y expr=\thisrow{b},x =a, col sep = comma]{Graphics/Data/Delta/LossH1Dual.csv};
\node[anchor=west] (source) at (axis cs:10e-1,9*10e-5){\color{blue}Training};
       \node (destination) at (axis cs:10e0,0.25*10e-1){};
       \draw[blue, ->](source)--(destination);
\addplot[very thick, mark=none,red,dashed ]table [y expr=\thisrow{b},x =a, col sep = comma]{Graphics/Data/Delta/ValH1Dual.csv};
\node[anchor=west] (source) at (axis cs:2*10e1,25*10e-3){\color{red}Validation};
       \node (destination) at (axis cs:10e2,10e-3){};
       \draw[red, ->](source)--(destination);

\end{axis}
\end{tikzpicture}
\end{center}\caption{Training and validation losses}\label{subfigMP5lossesEvolve}
\end{subfigure}
\begin{subfigure}[b]{0.33\textwidth}
\begin{center}
\begin{tikzpicture}
\begin{axis}[
  xlabel=$\sqrt{\mathcal{L}_{V^*}}$,
  ylabel=$H^1$-error,
height=4cm,
width=\textwidth,
xmode=log,
ymode=log,
legend style={at={(0.02,0.98)},anchor=north west}]
\addplot [only marks, blue, mark size=2pt, mark=+]table [y expr=sqrt(\thisrow{c}),x expr=sqrt(\thisrow{b}), col sep = comma]{Graphics/Data/Delta/LossH1Dual.csv};
\addlegendentry{Training}
\addplot[only marks, red, mark size=2pt, mark=x]table [y expr=sqrt(\thisrow{c}), ,x expr=sqrt(\thisrow{b}), col sep = comma]{Graphics/Data/Delta/ValH1Dual.csv};
\addlegendentry{Validation}
\node (s1) at (axis cs:10e-5,0.9*10e-5){};
\node (s2) at (axis cs:10e2,0.9*10e2){};
\draw (s1)--(s2);
\end{axis}
\end{tikzpicture}
\end{center}\caption{Loss/error correlation}
\label{subfigMP5correl}
\end{subfigure}\end{minipage}}
\caption{Model Problem 4: Loss and relative $H^1$-error during training}
\label{figMP5Losses}
\end{figure}

We see qualitatively in \Cref{figMP5sols} that we approximate well the exact solution, with absolute pointwise errors remaining of order $10^{-3}$. The exact solution is not $C^1$, and its derivative admits a jump discontinuity at $x=\frac{\pi}{2}$, thus we observe a Gibbs phenomenon-like error in the gradient of our obtained solution, which is to be expected as we are approximating with smooth trial functions. 

With our DFR method, we obtain a relative $L^2$ error of $0.039 \%$, and relative $H^1$-error of $3.60\%$. {\it EarlyStopping} halted training at $9610$ iterations. \Cref{subfigMP5lossesEvolve} shows that that overfitting develops towards the end of the training process. When this overfitting occurs, in \Cref{subfigMP5correl} we observe that the relative $H^1$-error has a sublinear dependency on the training loss; however, the square root of the validation loss and $H^1$ relative error exhibit a strong linear correlation. In particular, we see that at the point where training was halted by the {\it EarlyStopping} callback, the $H^1$-error had reached a plateau.

\subsubsection{Model Problem 5: Nonlinear}\label{ModelProblemNonlinear}
We take $V=H^1_0(0,\pi)$, and aim to find $u\in V$ such that 
\begin{equation}\begin{split}
&\int_0^\pi \left(u'(x)+\frac{1}{2}\sin(u'(x))\right)v'(x)+f(x)v(x)+u(x)v(x)+u(x)^3v(x)\,dx =0
\end{split}
\end{equation}
for all $v\in V$, where $f$ is obtained via the manufactured solution 
\begin{equation}
u^*(x)=5x\left(x-\frac{\pi}{2}\right)\tanh\left(5\left(x-\pi\right)\right).
\end{equation}
This problem admits a unique solution as it corresponds to the Euler-Lagrange equation of the strictly convex integral functional given by 
\begin{equation}
\mathcal{F}(u)=\int_0^\pi \frac{1}{2}|u'(x)|^2-\frac{1}{2}\cos(u'(x))+f(x)u(x)+\frac{1}{2}u(x)^2+\frac{1}{4}u(x)^4\,dx .
\end{equation}

The ODE is nonlinear, and thus the classical error estimate \eqref{eqUpLowerBd} does not directly apply. However, as commented in \Cref{remarkNonlinearity}, for a candidate solution close to the exact solution, the equation can be interpreted as a small perturbation of a linear problem. Consequently, we expect to see a linear regime towards the end of the training. 

\begin{figure}[H]
\begin{center}\centerline{\begin{minipage}{1.2\textwidth}
\begin{subfigure}[b]{0.33\textwidth}
\begin{center}
\begin{tikzpicture}
\begin{axis}[
  xlabel=$x$,
  ylabel=$u$,
xmin = 0,
xmax=3.1416,
height=4.5cm,
width=\textwidth]
\addplot [mark=none,very thick] table [y=b, x=a, col sep = comma]{Graphics/Data/Nonlinear3/SolExact.csv};
\node[anchor=west] (source) at (axis cs:1.0,-0.2){$u^*$};
       \node (destination) at (axis cs:1.,9){};
       \draw[->](source)--(destination);
\addplot [mark=none, blue, very thick,dashed ]table [y=b, x=a, col sep = comma]{Graphics/Data/Nonlinear3/SolDual.csv};
\node[anchor=west] (source) at (axis cs:1.8,-0.2){\color{blue}$u_{V^*}$};
       \node (destination) at (axis cs:2,-7){};
       \draw[blue, ->](source)--(destination);

\end{axis}
\end{tikzpicture}
\caption{ Approximate solution}
\end{center}
\end{subfigure}
\begin{subfigure}[b]{0.33\textwidth}
\begin{center}
\begin{tikzpicture}
\begin{axis}[
  xlabel=$x$,
  ylabel=$u-u^*$,
xmin = 0.001,
xmax=3.1414,
height=4.5cm,
width=\textwidth]
\addplot [mark=none, blue, very thick ]table [y=b, x=a, col sep = comma]{Graphics/Data/Nonlinear3/ErDual.csv};
\node[anchor=west] (source) at (axis cs:0.8,-5*10e-4){\color{blue}$\epsilon_{V^*}$};
       \node (destination) at (axis cs:1,10e-5){};
       \draw[blue, ->](source)--(destination);

\end{axis}
\end{tikzpicture}
\caption{Error functions}
\end{center}
\end{subfigure}
\begin{subfigure}[b]{0.33\textwidth}
\begin{center}
\begin{tikzpicture}
\begin{axis}[
  xlabel=$x$,
  ylabel=$(u-u^*)'$,
xmin = 0,
xmax=3.1416,
height=4.5cm,
width=\textwidth]
\addplot [mark=none, blue, very thick ]table [y=b, x=a, col sep = comma]{Graphics/Data/Nonlinear3/GradErDual.csv};
\node[anchor=west] (source) at (axis cs:0.5,0.05){\color{blue}$\epsilon_{V^*}'$};
       \node (destination) at (axis cs:1,0.2*0.025){};
       \draw[blue, ->](source)--(destination);
\end{axis}
\end{tikzpicture}
\caption{Errors in the gradient}
\end{center}
\end{subfigure}
\end{minipage}}
\caption{ Model Problem 5. Obtained solution and error}
\label{figMP6Sols}
\end{center}
\end{figure}
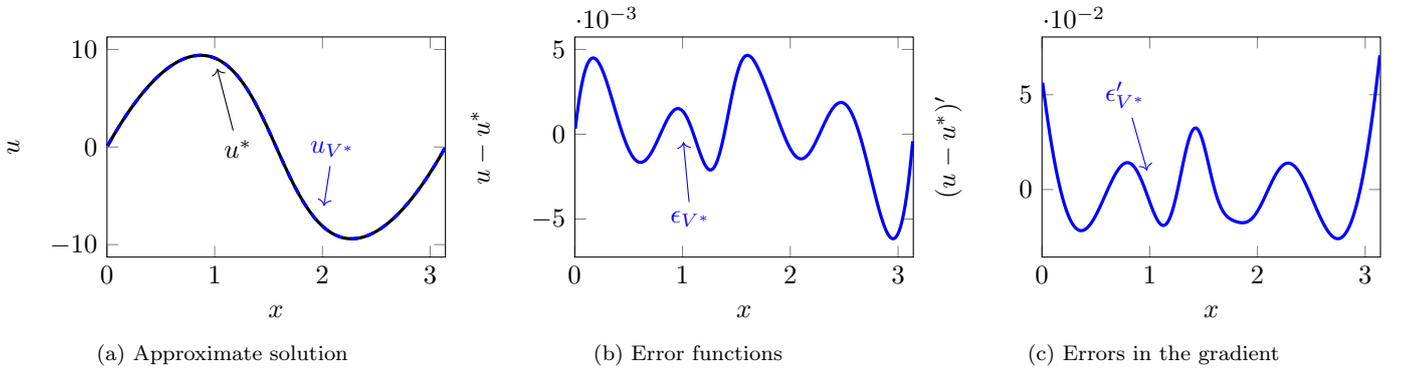

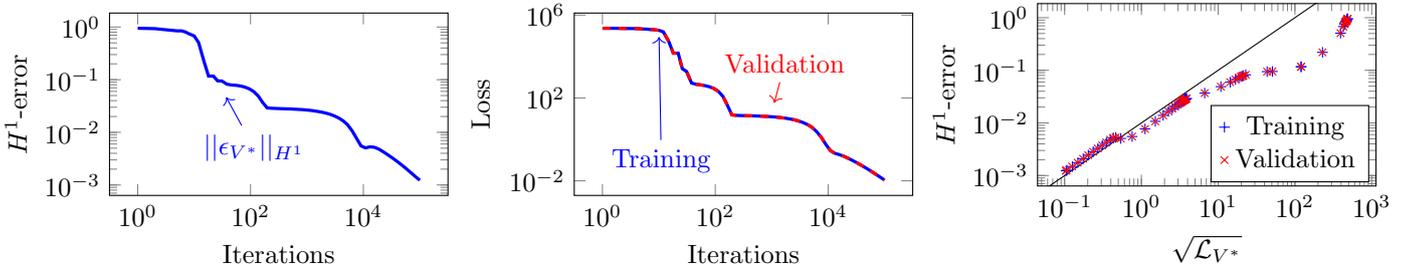
\begin{figure}[H]\centerline{\begin{minipage}{1.2\textwidth}
\begin{subfigure}[b]{0.33\textwidth}
\begin{center}
\begin{tikzpicture}
\begin{axis}[
  xlabel=Iterations,
  ylabel=$H^1$-error,
height=4cm,
width=\textwidth,
xmode=log,
ymode=log]
\addplot [mark=none, blue, very thick ]table [y expr=sqrt(\thisrow{c}),x=a, col sep = comma]{Graphics/Data/Nonlinear3/LossH1Dual.csv};
\node[anchor=west] (source) at (axis cs:10e0,5*10e-4){\color{blue}$||\epsilon_{V^*}||_{H^1}$};
       \node (destination) at (axis cs:3*10e0,7*10e-3){};
       \draw[blue, ->](source)--(destination);

\end{axis}
\end{tikzpicture}
\end{center}
\end{subfigure}
\begin{subfigure}[b]{0.33\textwidth}
\begin{center}
\begin{tikzpicture}
\begin{axis}[
  xlabel=Iterations,
  ylabel=Loss,
height=4cm,
width=\textwidth,
xmode=log,
ymode=log]
\addplot [very thick, mark=none,blue]table [y expr=\thisrow{b},x =a, col sep = comma]{Graphics/Data/Nonlinear3/LossH1Dual.csv};
\node[anchor=west] (source) at (axis cs:10e-1,9*10e-3){\color{blue}Training};
       \node (destination) at (axis cs:10e0,0.25*10e5){};
       \draw[blue, ->](source)--(destination);
\addplot[very thick, mark=none,red,dashed ]table [y expr=\thisrow{b},x =a, col sep = comma]{Graphics/Data/Nonlinear3/ValH1Dual.csv};
\node[anchor=west] (source) at (axis cs:10e1,5*9*10e1){\color{red}Validation};
       \node (destination) at (axis cs:10e2,4*5*10e-1){};
       \draw[red, ->](source)--(destination);

\end{axis}
\end{tikzpicture}
\end{center}
\end{subfigure}
\begin{subfigure}[b]{0.33\textwidth}
\begin{center}
\begin{tikzpicture}
\begin{axis}[
  xlabel=$\sqrt{\mathcal{L}_{V^*}}$,
  ylabel=$H^1$-error,
height=4cm,
width=\textwidth,
xmode=log,
ymode=log,
legend style={at={(0.98,0.02)},anchor=south east}]
\addplot [only marks, blue, mark size=2pt, mark=+]table [y expr=sqrt(\thisrow{c}),x expr=sqrt(\thisrow{b}), col sep = comma]{Graphics/Data/Nonlinear3/LossH1Dual.csv};
\addlegendentry{Training}
\addplot[only marks, red, mark size=2pt, mark=x]table [y expr=sqrt(\thisrow{c}), ,x expr=sqrt(\thisrow{b}), col sep = comma]{Graphics/Data/Nonlinear3/ValH1Dual.csv};
\addlegendentry{Validation}
\node (s1) at (axis cs:10e-5,10e-5/100){};
\node (s2) at (axis cs:10e2,10e2/100){};
\draw (s1)--(s2);
\end{axis}
\end{tikzpicture}
\end{center}
\end{subfigure}\end{minipage}}
\caption{Model Problem 5: Loss and relative $H^1$-error during training}
\label{figMP6Losses}
\end{figure}

After $10^5$ iterations, we obtain a relative $H^1$-error of $0.117\%$ and relative $L^2$ error of $0.036\%$. \Cref{figMP6Sols} shows that we have a good approximation of the exact solution, and the pointwise error is of order $10^{-3}$, and pointwise error in the gradient is of order $10^{-2}$. In \Cref{figMP6Losses} we observe that in early training we have a non-linear and slightly non-monotonic relationship between the square root of the loss and $H^1$-error; however, once we reach a relative error of around $10^{-2}$, we recover a linear regime with proportional dependence between the two metrics in accordance with the theory.

\subsubsection{Model Problem 6 - Discontinuous parameters in 2D}\label{ModelProblem2D}

Let $\Omega=[0,\pi]\times[0,\pi]$. We take $\Gamma_D$ to be three edges of $\partial\Omega$ corresponding to $x_1=0,\pi$ and $x_2=0$, and $\Gamma_N$ the edge corresponding to $x_2=\pi$. We aim to find the weak solution $u\in V$ to the equation 
\begin{equation}
\int_\Omega\sigma(x)\nabla u(x)\cdot\nabla v(x)+f(x)v(x)\,dx-\int_{0}^\pi v(x_1,\pi)\pi(x_1-\pi)x_1(1-\pi)\,dx_1 =0
\end{equation}
for all $v\in V$, where 
\begin{equation}
\begin{split}
\sigma(x)=&\left\{\begin{array}{c c}
2 & \left|x-\left(\frac{\pi}{2},\frac{\pi}{2}\right)\right|<1\\
1& \left|x-\left(\frac{\pi}{2},\frac{\pi}{2}\right)\right|\geq 1\\
\end{array}\right.,\\
f(x)=&\Delta\left(\left(x_1-\pi\right)\left(x_2-\pi\right)x_1x_2\left(1-\left|x-\left(\frac{\pi}{2},\frac{\pi}{2}\right)\right|^2\right)\right).
\end{split}
\end{equation}
The exact solution is given by 
\begin{equation}
u^*(x)=\frac{1}{\sigma(x)}\left(x_1-\pi\right)\left(x_2-\pi\right)x_1x_2\left(1-\left|x-\left(\frac{\pi}{2},\frac{\pi}{2}\right)\right|^2\right).
\end{equation}

We use an NN basis of five hidden layers each containing ten neurons and tanh activation function. 200x200 points are used for integration in the training loss, and 274x274 for validation. We have an initial learning rate of $10^{-2}$ with Adam, and run for $10^5$ iterations

\begin{figure}
\begin{subfigure}[b]{0.45\textwidth}\begin{center}
\includegraphics[width=\textwidth]{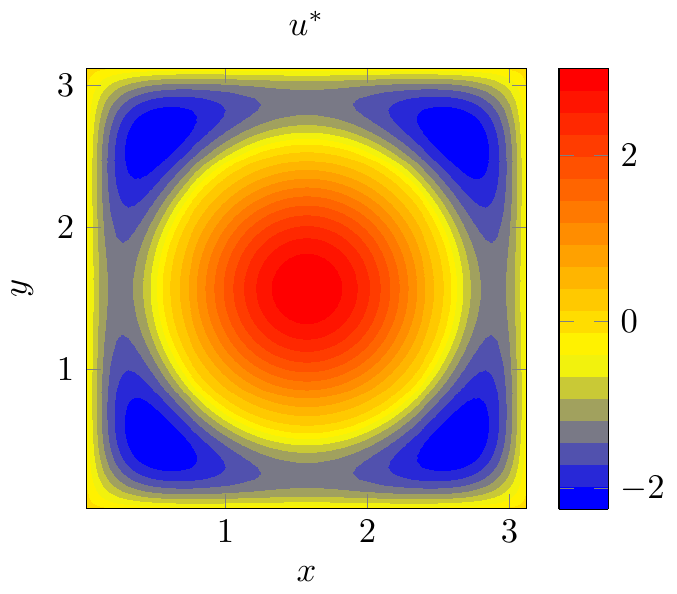}
\end{center}
\end{subfigure}
\begin{subfigure}[b]{0.45\textwidth}\begin{center}
\includegraphics[width=\textwidth]{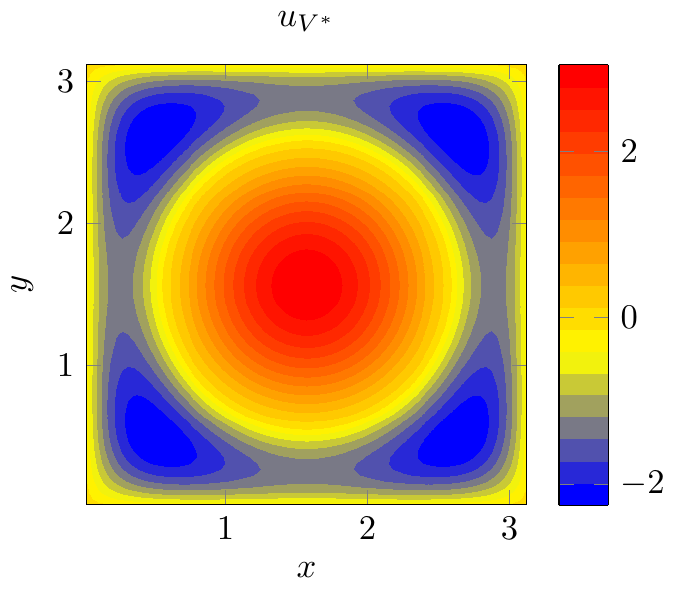}
\end{center}
\end{subfigure}

\begin{subfigure}[b]{0.45\textwidth}
\begin{center}
\includegraphics[width=\textwidth]{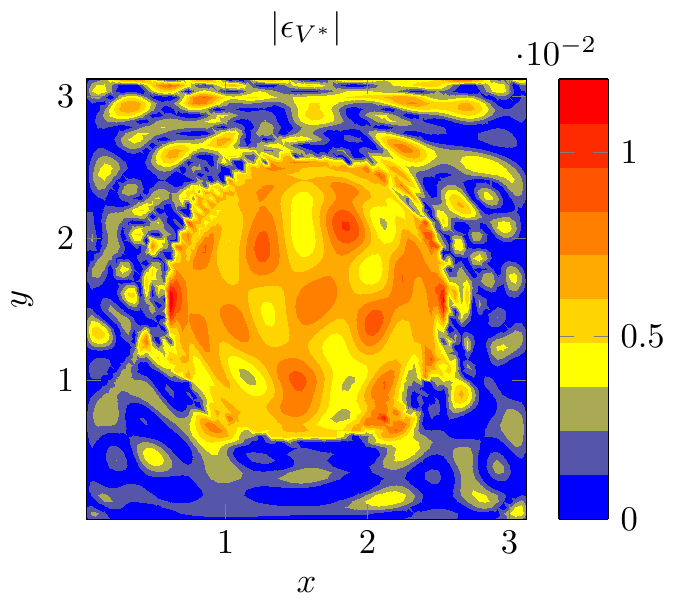}
\end{center}
\end{subfigure}
\begin{subfigure}[b]{0.45\textwidth}
\begin{center}\begin{minipage}[b]{0.95\textwidth}
\includegraphics[width=\textwidth]{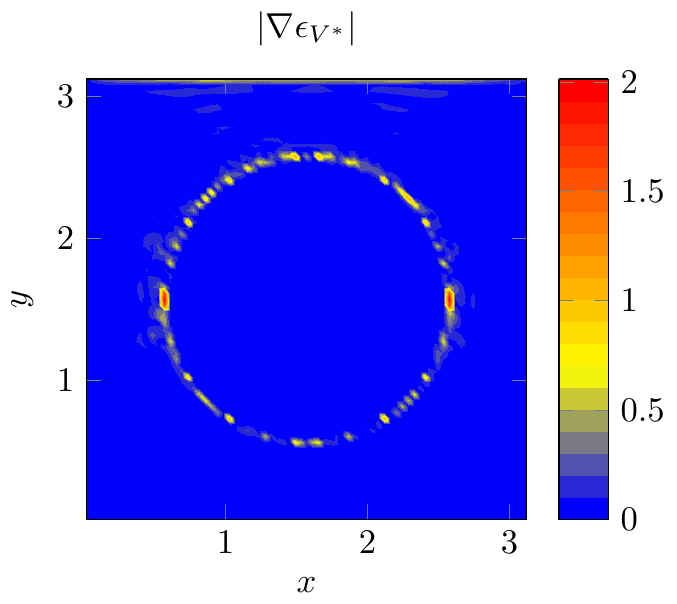}\end{minipage}
\end{center}
\end{subfigure}
\caption{Model Problem 6: Obtained solution and error}
\label{figMP4Sols}
\end{figure}

\begin{figure}[H]\centerline{\begin{minipage}{1.2\textwidth}
\begin{subfigure}[b]{0.33\textwidth}
\begin{center}
\begin{tikzpicture}
\begin{axis}[
  xlabel=Iterations,
  ylabel=$H^1$-error,
height=4cm,
width=\textwidth,
xmode=log,
ymode=log]
\addplot [mark=none, blue, very thick ]table [y expr=sqrt(\thisrow{c}),x=a, col sep = comma]{Graphics/Data/DiscSigma2d2/LossH1Dual.csv};
\node[anchor=west] (source) at (axis cs:10e-1,10e-2){\color{blue}$||\epsilon_{V^*}||_{H^1}$};
       \node (destination) at (axis cs:9*10e0,3*10e-2){};
       \draw[blue, ->](source)--(destination);

\end{axis}
\end{tikzpicture}
\end{center}
\caption{$H^1$-relative error during training}
\label{subfigMP4ErrorEvolve}
\end{subfigure}
\begin{subfigure}[b]{0.33\textwidth}
\begin{center}
\begin{tikzpicture}
\begin{axis}[
  xlabel=Iterations,
  ylabel=Loss,
height=4cm,
width=\textwidth,
xmode=log,
ymode=log]
\addplot [very thick, mark=none,blue]table [y expr=\thisrow{b},x =a, col sep = comma]{Graphics/Data/DiscSigma2d2/LossH1Dual.csv};
\node[anchor=west] (source) at (axis cs:7*10e1,1.44*10e1){\color{blue}Training};
       \node (destination) at (axis cs:0.8*10e2,0.09*10e1){};
       \draw[blue, ->](source)--(destination);
\addplot[very thick, mark=none,red,dashed ]table [y expr=\thisrow{b},x =a, col sep = comma]{Graphics/Data/DiscSigma2d2/ValH1Dual.csv};
\node[anchor=west] (source) at (axis cs:2*10e-1,10e-1){\color{red}Validation};
       \node (destination) at (axis cs:0.5*10e2,0.25*10e1){};
       \draw[red, ->](source)--(destination);

\end{axis}
\end{tikzpicture}
\end{center}
\caption{Training and validation losses}
\label{subfigMP4LossEvolve}
\end{subfigure}
\begin{subfigure}[b]{0.33\textwidth}
\begin{center}
\begin{tikzpicture}
\begin{axis}[
  xlabel=$\sqrt{\mathcal{L}_{V^*}}$,
  ylabel=$H^1$-error,
height=4cm,
width=\textwidth,
xmode=log,
ymode=log,
legend style={at={(0.02,0.98)},anchor=north west}]
\addplot [only marks, blue, mark size=2pt, mark=+]table [y expr=sqrt(\thisrow{c}),x expr=sqrt(\thisrow{b}), col sep = comma]{Graphics/Data/DiscSigma2d2/LossH1Dual.csv};
\addlegendentry{Training}
\addplot[only marks, red, mark size=2pt, mark=x]table [y expr=sqrt(\thisrow{c}), ,x expr=sqrt(\thisrow{b}), col sep = comma]{Graphics/Data/DiscSigma2d2/ValH1Dual.csv};
\addlegendentry{Validation}
\node (s1) at (axis cs:10e-5,0.035*10e-5){};
\node (s2) at (axis cs:10e2,0.035*10e2){};
\draw (s1)--(s2);
\end{axis}
\end{tikzpicture}
\end{center}\caption{Correlation between loss and error}\label{subfigMP4LossError}
\end{subfigure}\end{minipage}}
\caption{Model Problem 6: Loss and relative $H^1$-error during training}
\label{figMP4Losses}
\end{figure}
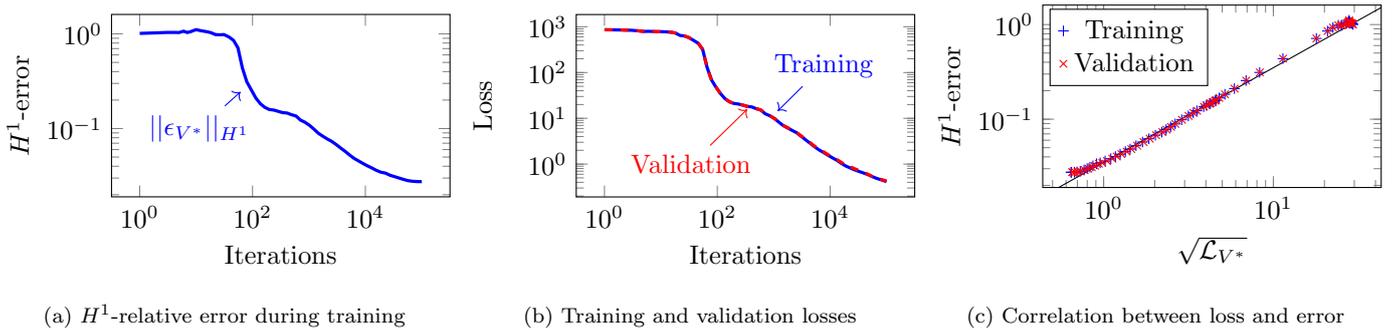
\Cref{figMP4Sols} shows that the method produces an accurate solution, with absolute pointwise errors remaining of the order $10^{-2}$. By numerical integration we observe a relative $H^1$-error of $2.7\%$ at the end of training. We observe more significant errors in $\nabla u_{V^*}$ near the ring of discontinuity in $\sigma$ and $\nabla u^*$, which is to be expected as we are approximating discontinuous functions with smooth functions. Outside of this ring-shaped region, however, the approximation of the gradient is generally good. \Cref{subfigMP4ErrorEvolve} shows a monotonic decay of the $H^1$-error during training, and \Cref{subfigMP4LossEvolve} shows that there is no overfitting present. Finally, \Cref{subfigMP4LossError} once again shows a linear relationship between the square root of the loss and the $H^1$-error.

\section{Conclusions}\label{secConc}

There are a wide class of PDEs in weak form, using $H^{1}$ as their space of test functions, such that the $H^{1}$-error of solutions can be controlled by the $H^{-1}$-norm of the PDE residual, as outlined in \Cref{propNonlinearEstimate}. We have developed a framework for implementing the $H^{-1}$ norm as a loss function to solve PDEs using NNs, which is numerically implemented via a spectral decomposition of the residual using DCT/DST to improve efficiency. We have numerically demostrated that in problems with sufficiently regular solutions, the method is comparable to the collocation and VPINNs methods; however, it shows a strong advantage when solutions lack $H^2$ regularity, in particular, when the PDE contains discontinuous material parameters or point sources. One may also use the proposed loss as a metric to assess the quality of approximate solutions, even if it is unused for optimisation.

In the absence of overfitting, we observe strong correlations between the training loss and $H^1$-error of candidate solutions. Moreover, overfitting is identified in our examples when divergence between the loss evaluated on a training and validation set occurs. This provides a strong advatange over the PINN and VPINN losses, which are inappropriate to use when solutions admit low regularity and may lead to erroneous results.

The DFR has several limitations that open the possibility for future research directions. First, our method suffers from the curse of dimensionality as one must perform DCT/DST in each coordinate direction. It may be possible to overcome this issue in higher dimensions by choosing more appropriate basis functions rather than tensor products of 1D basis sets. Second, our use of DCT/DST to numerically evaluate the dual norm naturally restricts our method to rectangular domains with appropriate boundary conditions on each face/edge. In arbitrary domains, one would need to find alternative basis functions and quadrature rules to numerically approximate the dual norm, which would be dependent on the particular geometry. Finally, our method approximates the $H^{-1}$ norm, which in certain PDEs such as the high-frequency Helmholtz equation, falls short at controlling the energy-norm error. To overcome this, one would need to find an appropriate basis to estimate the correct norm on the dual space via the series expansion \eqref{eqDualNormSeries}.

\section{Acknowledgements}
Jamie M. Taylor is supported by the Basque Government through the BERC 2018-2021 program and by the Spanish State Research Agency through BCAM Severo Ochoa excellence accreditation SEV-2017-0718 and through project PID2020-114189RB-I00 funded by Agencia Estatal de Investigaci\'on (PID2020-114189RB-I00 / AEI / 10.13039/501100011033).
David Pardo and Ignacio Muga have received funding from: the European Union's Horizon 2020 research and innovation program under the Marie Sklodowska-Curie grant agreement No 777778 (MATHROCKS). David Pardo has received funding from: the Spanish Ministry of Science and Innovation projects with references TED2021-132783B-I00, PID2019-108111RB-I00 (FEDER/AEI) and PDC2021-121093-I00 (AEI/Next Generation EU), the “BCAM Severo Ochoa” accreditation of excellence (SEV-2017-0718); and the Basque Government through the BERC 2022-2025 program, the three Elkartek projects 3KIA (KK-2020/00049), EXPERTIA (KK-2021/00048), and SIGZE (KK-2021/00095), and the Consolidated Research Group MATHMODE (IT1456-22) given by the Department of Education

\bibliography{bib}
\appendix

\section{The Laplacian basis}
\label{appLaplacian}

The following results are classical, with more detailed discussion available, for example, in \cite[Chapter 6]{brezis2011functional} or \cite{davies1996spectral}. In particular, Corollary 4.2.3 and Theorems 4.5.1 and 6.3.1. We include this discussion in a relatively self-contained framework for completeness.

Let $\Omega\subset\mathbb{R}^N$ be a bounded domain, with $\Gamma_D,\Gamma_N$ disjoint subsets of $\partial\Omega$ such that $\partial\Omega=\overline{\Gamma_D}\cup\overline{\Gamma_N}$. Take $V$ to be the space $V=\{v\in H^1(\Omega):v|_{\Gamma_D}=0\}$. We consider an orthogonal basis for $V$ given by the eigenvectors of the operator $1-\Delta$ on $V$ with boundary condition $\frac{\partial u}{\partial \nu}=0$ on $\Gamma_N$, that is, a homogeneous Neumann condition on $\Gamma_N$ and homogeneous Dirichlet condition on $\Gamma_D$. 

First, we show that such a basis exists. We first define the solution operator $T:L^2(\Omega)\to L^2(\Omega)$ to be the operator taking $f\in L^2(\Omega)$ to the unique solution $u\in V$ of  
\begin{equation}\label{eqWeakFormLap}
\int_\Omega \nabla u\cdot\nabla v +uv\,dx=\int_\Omega fv\,dx.
\end{equation}
for all $v\in V$. We remark that this is equivalent to $\langle Tf,v\rangle_{H^1}=\langle f,v\rangle_{L^2}$ for all $v\in V$. $T$ is a symmetric and positive definite linear map: For any $f_1,f_2\in L^2(\Omega)$, from the weak-formulation \eqref{eqWeakFormLap}, as $Tf_1,Tf_2\in V$, we have that 
\begin{equation}\label{eqInvLapIdentity}
\langle Tf_1,f_2\rangle_{L^2}=\langle Tf_1,Tf_2\rangle_{H^1}=\langle f_1,Tf_2\rangle_{L^2}.
\end{equation}
 Furthermore, by the classical Lax-Migram result, we have that $||Tf||_{H^1}\leq ||f||_{L^2}$. In particular, we have that $T$ is a compact symmetric operator from $L^2(\Omega)$ to itself, and thus, by classical spectral theory, this implies that $T$ admits a decreasing countable sequence of positive eigenvalues that converges monotonically to zero. For notational convenience, we consider their inverses, so that $T$ admits eigenvalues $\lambda_k^{-1}>0$ where $\lambda_k$ is a positive, monotonically increasing sequence with $\lambda_k^{-1}\to\infty$. The eigenvalues have corresponding eigenvectors $\varphi_k$, where $(\varphi_k)_{k=1}^\infty$ forms an orthonormal basis of $L^2(\Omega)$. The weak formulation \eqref{eqInvLapIdentity} implies that for any $v\in V$, 
\begin{equation}\label{eqEVSH1toL2}
\lambda_k^{-1}\langle \varphi_k,v\rangle_{H^1}=\langle T\varphi_k,v\rangle_{H^1}=\langle \varphi_k,v\rangle_{L^2}
\end{equation} 
In particular, as $\lambda_k\neq 0$, the $L^2$-orthogonality of the sequence $(\varphi_k)_{k=1}^\infty$ also implies $H^1$-orthogonality. By taking $v=\varphi_k$ in \eqref{eqEVSH1toL2}, we also see that $||\varphi_k||_{H^1}^2=\lambda_k$.

We show by contradiction that $\varphi_k$ also forms a {\it basis} of $V$, and not just an orthogonal set. If $\varphi_k$ were not a basis, there would exist some $v\in V\setminus \{0\}$ such that $\langle\varphi_k,v\rangle_{H^1}=0$ for all $k$. In light of \eqref{eqEVSH1toL2}, we must therefore have that $\langle \varphi_k,v\rangle_{L^2}=0$ for all $k$. This contradicts that $(\varphi_k)_{k=1}^\infty$ is a basis for $L^2(\Omega)$.

\section{Estimation via Fast (Co)Sine Transforms}
\label{appFFT}

Discrete Sine/Cosine Transforms (DST/DCT) are efficient methods, based on the Fast Fourier Transform (FFT), to decompose a finite input vector of dimension $N$ into $N$ sine or cosine waves with given boundary conditions. There are numerous variations corresponding to different boundary conditions, and within this work we focus on the type-II and type-IV transforms \cite[Section 4.2]{britanak2010discrete}. We use the notation DST-II, DST-IV to refer to the type-II and type-IV DST, respectively, and DCT-II and DCT-IV to refer to the type-II and type-IV DCT, respectively, which is employed in the summary of basis functions in \Cref{tableBases}.

As the DST/DCT are linear operations between two $N$-dimensional vector spaces, each may be represented by an $N\times N$ matrix. We represent the type-II and type-IV DST via matrices $S^{II}$ and $S^{IV}$, respectively, and similarly the type-II and type-IV DCT via $C^{II}$ and $C^{IV}$. Each matrix is indexed by $k,n=0,...,N-1$. For an integrable function $g$, we may approximate integrals via the following four relationships:
\begin{equation}\label{eqFFTs}
\begin{split}
\int_0^{\pi} g(x)\sin(kx)\,dx \approx \,&\mathcal{S}_{N,k}^{II}(g):=\sum_{n=0}^{N-1} \frac{\pi}{\sqrt{2N}}S^{II}_{k-1,n}g\left(\frac{2n+1}{2N}\pi\right),\\
\int_0^\pi g(x)\sin\left(\left(k-\frac{1}{2}\right)x\right)\,dx\approx &\,\mathcal{S}^{IV}_{N,k}(g):=\sum\limits_{n=0}^{N-1}\frac{\pi}{\sqrt{2N}}S^{IV}_{k-1,n}g\left(\frac{2n+1}{2N}\pi\right),\\
\int_0^{\pi} g(x)\cos(kx)\,dx \approx \,&\mathcal{C}^{II}_{N,k}(g):=\sum_{n=0}^{N-1} \frac{\pi}{\sqrt{2N}}C^{II}_{k-1,n}g\left(\frac{2n+1}{2N}\pi\right),\\
\int_0^\pi g(x)\cos\left(\left(k-\frac{1}{2}\right)x\right)\,dx\approx &\,\mathcal{C}^{IV}_{N,k}(g):=\sum\limits_{n=0}^{N-1}\frac{\pi}{\sqrt{2N}}C^{IV}_{k-1,n}g\left(\frac{2n+1}{2N}\pi\right).\\
\end{split}
\end{equation}
In each case, the approximation corresponds to a mid-point integration rule, that is, 
\begin{equation}
\int_0^\pi f(x)\,dx \approx \frac{\pi}{N}\sum\limits_{n=0}^{N-1} f\left(\frac{2n+1}{2N}\pi\right).
\end{equation}
The matrices in the transformations are defined by:
\begin{equation}
\begin{split}
(S^{II}_N)_{kn}:=&\sqrt{\frac{2}{N}}\epsilon_k\sin\left(\frac{\pi}{N}\left(n+\frac{1}{2}\right)(k+1)\right),\\
(S^{IV}_N)_{kn}:=&\sqrt{\frac{2}{N}}\sin\left(\frac{\pi}{N}\left(n+\frac{1}{2}\right)\left(k+\frac{1}{2}\right)\right),\\
(C^{II}_N)_{kn}:=& \sqrt{\frac{2}{N}}\epsilon_k'\cos\left(\frac{\pi}{N}\left(n+\frac{1}{2}\right)k\right),\\
(C^{IV}_N)_{kn}:=& \sqrt{\frac{2}{N}}\cos\left(\frac{\pi}{N}\left(n+\frac{1}{2}\right)\left(k+\frac{1}{2}\right)\right),
\end{split}
\end{equation}
for $k,n=0,...,N-1$ and where 

\begin{equation}
\begin{split}
\epsilon_k=&\left\{\begin{array}{c c}
1 & k\neq N-1,\\
\frac{1}{\sqrt{2}} & k=N-1.
\end{array}\right. \\
\epsilon_k'=&\left\{\begin{array}{c c}
1 & k\neq 0,\\
\frac{1}{\sqrt{2}} & k=0.
\end{array}\right.\\
\end{split}
\end{equation}

For both $X=II$, $X=IV$, the matrices $C^{X}_N$ and $S^{X}_N$ are related as follows. Let $D$ denote the diagonal matrix with diagonal entries $D_{kk}=(-1)^k$ for $k=0,...,N-1$, and $J$ denote the matrix which reverses the order of a vector, so that $JY=(y_{N-1},y_{N-2},...,y_1,y_0)$. Then, 
\begin{equation}
S^{X}_N=JC^{X}_ND
\end{equation}

Under this particular normalisation, the corresponding transformation matrices are orthogonal, that is, each of them satisfies $M^T=M^{-1}$.

\section{Nonlinear equations}

\label{appNonlinear}

We now prove \Cref{propNonlinearEstimate}. We do so by considering a linearisation, at which point we may invoke results on the linear theory. The reader is directed to \cite[Chapter 7]{ciarlet2013linear} for definitions and properties related to the differentiability of functions between Banach spaces, however we include below some definitions for completeness.

Given $u,w\in U$, we define the directional derivative $\delta_w\mathcal{R}(u)\in V^*$ via 
\begin{equation}
\delta_w\mathcal{R}(u):=\lim\limits_{\tau\to 0}\frac{\mathcal{R}(u+\tau w)-\mathcal{R}(u)}{\tau},
\end{equation}
when it exists. Furthermore, we state that $\mathcal{R}$ is Gateaux differentiable at $u$ if $\delta_w\mathcal{R}(u)$ exists for all $w\in U$, and the map $U\ni w\mapsto \delta_w\mathcal{R}(u)\in V^*$ defines a continuous linear function. 

The proof of \Cref{propNonlinearEstimate} reduces to two lemmas, corresponding to the upper and lower bounds.
\begin{lemma}
For every $\epsilon>0$, there exists $\delta_1>0$ such that for all $u\in V$ with $||u-u^*||_{U}<\delta_1$, 
$$||\mathcal{R}(u)||_{V^*}\geq \frac{\gamma}{1+\epsilon}||u-u^*||_{U}.$$
\end{lemma}
\begin{proof}
We consider $||u-u^*||_{U}<r_0$. For brevity, we denote $w=u-u^*$. By Taylor's theorem with explicit remainder, we can estimate 
\begin{equation}
\begin{split}
\left|\left|\mathcal{R}(u)-\mathcal{R}(u^*)-\delta_w\mathcal{R}(u^*)\right|\right|_{V^*}=& \left|\left|\int_0^1\delta_w\mathcal{R}(u^*+t(u-u^*))\,dt - \delta_w\mathcal{R}(u^*)\right|\right|_{V^*}\\
\leq&  \int_0^1\left|\left|\delta_w\mathcal{R}(u^*+t(u-u^*)) - \delta_w\mathcal{R}(u^*)\right|\right|_{V^*}\,dt\\
\leq & L\int_0^1||w||_U||u-u^*||_U\,dt=L||u-u^*||_U^2.
\end{split}
\end{equation}
Thus, we may define the remainder $R:V\to V^*$ via 
\begin{equation}
R(u)=\mathcal{R}(u)-\mathcal{R}(u^*)-\delta_w\mathcal{R}(u^*),
\end{equation}
which satisfies $||R(u)||_{V^*}\leq L||u-u^*||_{U}^2$. Via the reverse triangle inequality, and noting that $\mathcal{R}(u^*)=0$, we estimate 
\begin{equation}\label{eqCoerLB}
||\mathcal{R}(u)||_{V^*}=||\delta_w\mathcal{R}(u^*)-R(u)||_{V^*}\geq \Big|||\delta_v\mathcal{R}(u^*)||_{V^*}-||R(u)||_{V^*}\Big|.
\end{equation}

As $\delta_w\mathcal{R}(u^*)$ is bounded below, $||\delta_w\mathcal{R}(u^*)||_{V^*}\geq \gamma||u-u^*||_{U}$. We observe that if $$||u-u^*||_U<\frac{\gamma\epsilon}{(1+\epsilon)L},$$ where we interpret the right-hand side as $+\infty$ if $L=0$. Then $$L||u-u^*||_{U}^2\leq \frac{\gamma\epsilon}{1+\epsilon}||u-u^*||_{U}$$ and thus 
\begin{equation}\label{eqCoerLB2}
\begin{split}
 \left|||\delta_w\mathcal{R}(u^*)||_{V^*}-||R(u)||_{V^*}\right|\geq & \gamma||w||_U-\frac{\gamma\epsilon}{(1+\epsilon)}||w||_{U}= \frac{\gamma}{(1+\epsilon)}||u-u^*||_{U}. 
\end{split}
\end{equation}
By taking $\delta_1=\min\left(r_0,\frac{\gamma\epsilon}{(1+\epsilon)L}\right)$ and combining equations \eqref{eqCoerLB} and \eqref{eqCoerLB2}, the result holds. 
\end{proof}

We now turn to the upper bound. 

\begin{lemma}
For every $1>\epsilon>0$, there exists $\delta_2>0$ such that if $||u-u^*||_{U}<\delta_2$, then 
\begin{equation}\label{eqUppBoundNL}
||\mathcal{R}(u)||_{V^*}\leq \frac{M}{1-\epsilon}||u-u^*||_U. 
\end{equation}
\end{lemma}
\begin{proof}
Again, we proceed by assuming that $||u-u^*||_U<r_0$ and invoking Taylor's theorem. Taking the remainder $R$ and $w=u-u^*$ as before, we observe that 
\begin{equation}\begin{split}
||\mathcal{R}(u)||_{V^*}=& \left|\left|\delta_w\mathcal{R}(u^*)+R(u)\right|\right|_{V^*}\\
\leq & \left|\left|\delta_w\mathcal{R}(u^*)\right|\right|_{V^*}+\left|\left|R(u)\right|\right|_{V^*}\\
\leq & M||u-u^*||_U+L||u-u^*||_{U}^2.
\end{split}
\end{equation}
Thus, if $||u-u^*||_U < \frac{M\epsilon}{L(1-\epsilon)}$, interpreting the right-hand side of the inequality as $+\infty$ if $L=0$, we have that $L||u-u^*||_{U}^2\leq \frac{M\epsilon}{(1-\epsilon)}||u-u^*||_{U}$ and 
\begin{equation}
||\mathcal{R}(u)||_{V^*}\leq \left(M+\frac{M\epsilon}{(1-\epsilon)}\right)||u-u^*||_{V^*}=\frac{M}{1-\epsilon}||u-u^*||_{V^*}. 
\end{equation}
By taking $\delta_2=\min\left(r_0,\frac{M\epsilon}{(1-\epsilon)L}\right)$ we complete the proof. 
\end{proof}

\end{document}